\documentclass[a4paper, 10pt, notitlepage]{article}

\usepackage{amsthm, amsmath, amssymb, latexsym}
\usepackage{mathrsfs,cite}

\theoremstyle{plain}
\newtheorem{theorem}{Theorem}[section]	
\newtheorem{lemma}[theorem]{Lemma}	
\newtheorem{corollary}[theorem]{Corollary}
\newtheorem{proposition}[theorem]{Proposition}

\theoremstyle{definition}
\newtheorem{definition}{Definition}
\newtheorem{example}{Example}

\theoremstyle{remark}
\newtheorem{remark}{Remark}

\DeclareMathOperator{\dom}{dom}
\DeclareMathOperator{\interior}{int}
\DeclareMathOperator{\argmin}{arg\,min}

\author{M.V. Dolgopolik}
\title{A Unifying Theory of Exactness of Linear Penalty Functions}

\begin{document}

\maketitle

\begin{abstract}
In this article we develop a theory of exact linear penalty functions that generalizes and unifies most of the results
on exact penalization existing in the literature. We discuss several approaches to the study of both locally and
globally exact linear penalty functions, and obtain various necessary and sufficient conditions for the exactness of 
a linear penalty function. We pay more attention than usual to necessary conditions that allows us to deeper
understand the exact penalty technique.
\end{abstract}

\section{Introduction}

Starting with the pioneering works by Eremin \cite{Eremin} and Zangwill \cite{Zangwill}, the method of exact linear
penalty functions has been advanced by many researchers (see~\cite{HanMangasarian, IoffeNSC,
Rosenberg,Pietrzykowski,Bertsekas,Mangasarian, Burke, Polyakova, WuBaiYangZhang, Antczak, Demyanov, Demyanov0,
DemyanovDiPilloFacchinei,DiPilloGrippo,DiPilloGrippo2,ExactBarrierFunc,Zaslavski,RubinovYang,DiPilloFacchinei_InCollect}
and references therein), and became a standard tool of constrained optimization. Innumerable applications, both
theoretical and practical, of this method to constrained extremum problems of almost all imaginable types proved the
efficiency of the exact penalization technique, and cannot be reviewed even in the scope of a book. We mention here only
the fact that exact penalty functions can be used to obtain both well-known and new optimality conditions for various
constrained optimization problems \cite{HanMangasarian, Burke, MengYang1, MengYang2}

In most of the articles devoted to the theory of exact linear penalty functions and its applications, a penalty function
is constructed and studied only for particular classes of optimization problems (or only a particular
class of penalty functions, such as penalty functions that include constraints via a distance function, is considered),
and its analysis is often based on the use of peculiarities of a specific setting. Thus, similar in nature results on
exact penalty functions are proved many times in different settings. The other feature of the works on exact linear
penalty functions is an almost complete lack of necessary conditions for exact penalization (apart from a
characterization of exactness via the calmness of a perturbed problem \cite{BurkeCalm}), which are very
important for deeper understanding of the method of exact penalty functions.

The main goal of this article is to develop a theory of exact \textit{linear} penalty functions that can unify and
generalize numerous results on exactness of linear penalty functions scattered in the literature. Furthemore, we aimed
to deepen the understanding of the mechanics of exact penalization by paying more attention, than usual, to necessary
conditions. This approach allowed us to obtain not only necessary conditions for exact penalization but also some
interesting new necessary and sufficient conditions that provide a complete characterization of the exactness of a
linear penalty function. 

A natural formulation of the general problem that is studied within the theory of exact linear penalty functions, and
that we adopt in this article, is as follows. For the given constrained optimization problem
$$
  \min f(x) \quad \text{subject to} \quad x \in M, \quad x \in A, \eqno (\mathcal{P})
$$
and the corresponding penalized problem
$$
  \min F_{\lambda}(x) \quad \text{subject to} \quad x \in A,
$$
study a relation between locally/globally optimal solutions of the initial problem $(\mathcal{P})$, and 
locally/globally optimal solutions of the penalized problem (which might not belong to the set of feasible points of the
problem $(\mathcal{P})$). Here $M$ and $A$ are nonempty subsets of a topological space $X$, 
$f \colon X \to \mathbb{R} \cup \{ + \infty \}$ and
$$
  F_{\lambda}(x) = f(x) + \lambda \varphi(x), \quad \lambda \ge 0,
$$
where $\varphi \colon X \to [0, +\infty)$ is \textit{the penalty term}, i.e. $\varphi$ is a given function such that 
$M = \{ x \in X \mid \varphi(x) = 0 \}$. Note that only a part of the constraints (namely, $x \in M$) is penalized, and
no assumptions are made on the penalty term $\varphi(x)$. Moreover, the fact that the objective function $f$ can take
infinite values allows one to include the results on exact penalty functions with barrier terms (see, for instance,
\cite{ExactBarrierFunc}) into the general theory. Thus, the formulation of the problem that we use is general enough to
include most of the results on exact linear penalty functions into the general theory.

The paper is organized as follows. Section~\ref{SectLocalTheory} is devoted to the study of ``local''
exactness of penalty functions. It includes a theorem on necessary and sufficient conditions for a penalty function
to be exact at a given locally optimal solution of the initial problem (Theorem~\ref{ThNSCLocMin}), as well as two
general techniques for obtaining sufficient (or necessary) conditions for local exact penalization. Both of these
techniques (error bounds and calmness of a perturbed problem) are well-known, and our exposition of this
subject is a straightforward generalization of many existing results. Various new necessary and sufficient conditions
for a penalty function to be globally exact under different assumptions on the domain of the penalty function are
studied in Section~\ref{SectGlobTheory}. This section also contains a description of several general approaches to the
study of globally exact penalty functions. Some general results on a very important from a practical point of view
problem of points of local minimum of a penalty function that do not belong to the set of feasible points of the initial
problem are briefly discussed in Section~\ref{SectLocMinOfPenFunc}.

Note that although many results presented in this paper are a direct generalization of existing results on exact
linear penalty functions, some concepts (such as nondegeneracy of a penalty function) and results (such as
Theorems~\ref{ThExPenFuncCompactSet}, \ref{ThExPenFiniteDim}, \ref{ThFinDimGenCase} and \ref{ThExactOnBoundedSets}, and
Lemma~\ref{LemmaEPF}) are completely new.

\section{Local Theory of Exact Penalty Functions}
\label{SectLocalTheory}

In this section, we discuss a relation between locally optimal solution of a constrained optimization problem and local
minimizers of a penalty function for this problem. If a locally optimal solution of the problem is a also a point of
local minimum of a penalty function, then the penalty function is said to be \textit{exact} at this point. We study
necessary and sufficient conditions for a penalty function to be exact at a locally optimal solution of a constrained
optimization problem, and present two general techniques for obtaining upper estimates of the least exact penalty
parameter. The first one is based on local error bounds, while the second technique is based on the concept of calmness
of a perturbed optimization problem.

\subsection{Exactness of Penalty Functions}

From this point onwards, let $X$ be a topological space, $M$, $A \subset X$ be nonempty sets, and 
$f \colon X \to \mathbb{R} \cup \{ +\infty \}$ be a given function. Throughout the paper we study the following
constrained optimization problem
$$
  \min f(x) \quad \text{subject to} \quad x \in M, \quad x \in A. \eqno (\mathcal{P})
$$
Denote by $\Omega = M \cap A$ the set of feasible points of this problem, and by $f^* = \inf_{x \in \Omega} f(x)$ 
the optimal value of the problem $(\mathcal{P})$.

\begin{remark}
Hereafter, we suppose that the restriction of $f$ on $\Omega$ is a proper function, i.e. there exists $x \in \Omega$
such that $f(x) < +\infty$. We also implicitly suppose that if $x^* \in \Omega$ is a locally optimal solution of 
the problem $(\mathcal{P})$, then $f(x^*) < +\infty$.
\end{remark}

Let $\varphi \colon X \to \mathbb{R}_+ = [0, +\infty)$ be a given nonnegative function such that
$$
  M = \{ x \in X \mid \varphi(x) = 0 \}.
$$
Some examples of such functions $\varphi$ are given further in the text. For any $\lambda \ge 0$ define the function
$$
  F_{\lambda}(x) = f(x) + \lambda \varphi(x), \quad \forall x \in X,
$$
that is called \textit{a} (\textit{linear}) \textit{penalty function} for the problem $(\mathcal{P})$. Note
that for any $x \in X$ the penalty function $F_\lambda$ is nondecreasing in $\lambda$.

\begin{remark}
To simplify the exposition of the subject, we suppose that the penalty term $\varphi$, unlike the objective function
$f$, takes only finite values; however, this assumption is not restrictive. Indeed, if $\varphi(x) = +\infty$ for some 
$x \notin \Omega$, then one can redefine $f$ and $\varphi$, so that $f(x) = +\infty$ and $\varphi(x) < +\infty$ for
these $x \notin \Omega$, without changing optimal solution of the problem $(\mathcal{P})$, and the penalty function. 
\end{remark}

Alongside the problem $(\mathcal{P})$, we consider the following penalized problem
$$
  \min F_{\lambda}(x) \quad \text{subject to} \quad x \in A.
$$
It is easy to see that if $x^* \in \Omega$ is a locally optimal solution of the above problem, then $x^*$ is a
locally optimal solution of the problem $(\mathcal{P})$. However, the converse statement is not true in the general
case, unless $x^*$ belongs to the interior $\interior \Omega$ of the set of feasible points $\Omega$.

\begin{definition}
Let $x^*$ be a locally optimal solution of the problem $(\mathcal{P})$. The penalty function $F_{\lambda}$ is said to
be (locally)\textit{exact} (or to have \textit{the exact penalty property}) at the point $x^*$, if there exists
$\lambda^* \ge 0$ such that $x^*$ is a point of local minimum of $F_{\lambda^*}$ on the set $A$.
\end{definition}

From the fact that the penalty function is nondecreasing in $\lambda$ it follows that if $F_{\lambda}$ is
exact at a point $x^*$, then for any $\mu \ge \lambda^*$ the point $x^*$ is a local minimizer of $F_{\mu}$ on $A$,
where $\lambda^* \ge 0$ is from the definition of the exact penalty property. Denote the greatest lower bound of all
such $\lambda^* \ge 0$ by $\lambda^*(x^*, f, \varphi)$ or simply by $\lambda^*(x^*)$, if the functions $f$ and $\varphi$
are fixed. The quantity $\lambda^*(x^*, f, \varphi)$ is called \textit{the least exact penalty parameter} (of the
penalty function $F_{\lambda}$) at the point $x^*$.

It is obvious that for any $\lambda > \lambda^*(x^*)$ the point $x^*$ is a local minimizer of the penalty function
$F_{\lambda}$ on $A$, and for all $\lambda \in [0, \lambda^*(x^*))$ the point $x^*$ is not a local minimizer of
$F_{\lambda}$ on $A$. 

Let us mention here two simple yet useful properties of penalty functions. For the sake of completeness we provide
proofs of these results.

\begin{proposition} \label{PrpUniformLocalPenal}
Let the penalty function $F_{\lambda}$ be exact at a locally optimal solution $x^*$ of 
the problem $(\mathcal{P})$. Then for any $\lambda_0 > \lambda^*(x^*)$ there exists a neighbourhood $U$ of $x^*$
such that
$$
  F_{\lambda}(x) \ge F_{\lambda}(x^*) \quad \forall x \in U \cap A \quad \forall \lambda \ge \lambda_0,
$$
i.e. the neighbourhood $U$ does not depend on $\lambda \ge \lambda_0$.
\end{proposition}

\begin{proof}
Choose an arbitrary $\lambda_0 > \lambda^*(x^*)$. By the definition of the least exact penalty parameter
$\lambda^*(x^*)$, the point $x^*$ is a local minimizer of $F_{\lambda_0}$ on the set $A$. Consequently, there exists a
neighbourhood $U$ of $x^*$ such that $F_{\lambda_0}(x^*) \le F_{\lambda_0}(x)$ for all $x \in U \cap A$. Then taking
into account the facts that $F_{\lambda}$ is nondecreasing in $\lambda$, and $F_{\lambda}(x^*) = f(x^*)$ for all 
$\lambda \ge 0$ one obtains the desired result.
\end{proof}

\begin{proposition}
Let $x^*$ be a locally optimal solution of the problem $(\mathcal{P})$, and $\psi \colon X \to \mathbb{R}_+$ be a
function such that $M = \{ x \in X \mid \psi(x) = 0 \}$ and $\psi \le \varphi$ in a neighbourhood of $x^*$. If the
penalty function $\Phi_{\lambda} = f + \lambda \psi$ is exact at $x^*$, then the penalty function $F_{\lambda}$ is also
exact at this point and
$$
  \lambda^*(x^*, f, \psi) \ge \lambda^*(x^*, f, \varphi).
$$
\end{proposition}

\begin{proof}
Let $\lambda_0 > \lambda^*(x^*, f, \psi)$ be arbitrary, and $U \subset X$ be a neighbourhood of $x^*$ such that 
$\psi(x) \le \varphi(x)$ for all $x \in U$. By the definition of $\lambda^*(x^*, f, \psi)$ there exists a neighbourhood
$V \subset U$ of $x^*$ such that $\Phi_{\lambda_0}(x) \ge \Phi_{\lambda_0}(x^*)$ for all $x \in V \cap A$. Then applying
the estimate $\psi \le \varphi$, and taking into account the fact that 
$\Phi_{\lambda_0}(x^*) = F_{\lambda_0}(x^*) = f(x^*)$ one obtains that 
$$
  F_{\lambda_0}(x) \ge \Phi_{\lambda_0}(x) \ge F_{\lambda_0}(x^*) \quad \forall x \in V \cap A.
$$
Therefore the penalty function $F_{\lambda}$ is exact and $\lambda^*(x^*, f, \varphi) \le \lambda_0$, which obviously
implies that $\lambda^*(x^*, f, \psi) \ge \lambda^*(x^*, f, \varphi)$.
\end{proof}

The proposition above can be reformulated as follows. If the penalty function $\Phi_{\lambda} = f + \lambda \psi$, where
$\psi$ is some local lower estimate of the function $\varphi$, is exact at a point $x^*$, then $F_{\lambda}$ is also
exact at this point, and the least exact penalty parameter of $F_{\lambda}$ at $x^*$ is not greater than the least exact
penalty parameter of $\Phi_{\lambda}$ at this point.

\begin{remark}[Nonsmoothness of Exact Penalty Functions]
Let $X$ be a normed space. Note than if $F_{\lambda}$ is exact at a locally optimal solution $x^*$ of the problem
$(\mathcal{P})$, then, in the general case, the function $\varphi$ must be nonsmooth at $x^*$.
In other words, if $\varphi$ is smooth, then $F_{\lambda}$ is not exact in the general case. 

Indeed, let $\varphi$ be smooth at $x^*$, and suppose for the sake of simplicity that $A = X$. Observe that $x^*$ is a
point of global minimum of $\varphi$ by virtue of the facts that $\varphi$ is non-negative, and $\varphi(x^*) = 0$.
Therefore $\varphi'(x^*) = 0$, where $\varphi'(x^*)$ is a G\^{a}teaux (or Fr\'{e}chet) derivative of $\varphi$ at
$x^*$. 

The penalty function $F_{\lambda}$ is exact at $x^*$ or, equivalently, $x^*$ is a point of local minimum of
$F_{\lambda}$ for any sufficiently large $\lambda \ge 0$. Therefore $x^*$ must satisfy a first order necessary
optimality condition for $F_{\lambda}$ for any $\lambda \ge 0$ large enough. Hence, with the use of the fact that 
$\varphi'(x^*) = 0$, and $\varphi$ is smooth (i.e. its derivative is continuous), one can show that $x^*$ must also
satisfy the same first order necessary optimality condition for the function $f$, which is usually not the case because
$x^*$ is a locally optimal solution of a constrained optimization problem.

Thus, nonsmoothness of $\varphi$ must occur, when $x^*$ is a locally optimal solution of the constrained problem
$(\mathcal{P})$, while is not a point of unconstrained local minimum of $f$, what happens quite often. 
\end{remark}

\subsection{Necessary and Sufficient Condition for Local Exact Penalization}

Let $x^*$ be a locally optimal solution of the problem $(\mathcal{P})$. Intuitively, for the point $x^*$ to be a local
minimizer of the penalty function $F_{\lambda}$ on $A$ it is necessary that the rate of increase of the function
$\varphi$ outside the set $M = \{ x \in X \mid \varphi(x) = 0 \}$ near $x^*$ is ``greater'' than the rate of decrease of
the function $f$ outside $M$ near $x^*$. In order to clarify this idea we introduce the following quantity. For any 
$x \in M$ set
$$
  \overline{\lambda}(x, f, \varphi) = 
  \limsup_{y \to x, y \in A \setminus M} \frac{f(x) - f(y)}{\varphi(y) - \varphi(x)}.
$$ 
If $x$ is not a limit point of $A \setminus M$, then $\overline{\lambda}(x, f, \varphi) = -\infty$ by definition. In the
case when $f$ and $\varphi$ are fixed we write $\overline{\lambda}(x)$ instead of $\overline{\lambda}(x, f, \varphi)$.

The quantity $\overline{\lambda}(x)$ characterizes the rate of decrease of the function $f$ on the set $A \setminus M$
near a point $x \in M$ with respect to the rate of increase of the function $\varphi$ on $A \setminus M$ near $x$. The
following theorem gives a necessary and sufficient condition for a penalty function to be locally exact in terms of
$\overline{\lambda}(x)$.

\begin{theorem} \label{ThNSCLocMin}
Let $x^*$ be a locally optimal solution of the problem $(\mathcal{P})$. For the penalty function 
$F_{\lambda} = f + \lambda \varphi$ to be exact at the point $x^*$ it is necessary and sufficient that 
$\overline{\lambda}(x^*) < + \infty$. Furthermore, if $F_{\lambda}$ is exact at the point $x^*$, then 
$\lambda^*(x^*) = \max\{ \overline{\lambda}(x^*), 0 \}$.
\end{theorem}

\begin{proof}
It is easy to see that if $x^*$ is not a limit point of $A \setminus M$, then $F_{\lambda}$ is exact at $x^*$ and
$\lambda^*(x^*) = 0$. Therefore we can suppose that $x^*$ is a limit point of the set $A \setminus M$.

Necessity. Fix $\lambda > \lambda^*(x^*)$. By the definition of the least exact penalty parameter $\lambda^*(x^*)$ there
exists a neighbourhood $U$ of $x^*$ such that
$$
  f(x^*) + \lambda \varphi(x^*) = F_{\lambda}(x^*) \le F_{\lambda}(y) = f(y) + \lambda \varphi(y) \quad
  \forall y \in U \cap A,
$$
which yields
$$
  \frac{f(x^*) - f(y)}{\varphi(y) - \varphi(x^*)} \le \lambda \quad \forall y \in U \cap (A \setminus M)
$$
(note that $\varphi(x^*) = 0$, since $x^* \in M$). Therefore $\overline{\lambda}(x^*) \le \lambda < + \infty$, and
$\overline{\lambda}(x^*) \le \lambda^*(x^*)$. 

Sufficiency. Fix an arbitrary $\lambda > \max\{ 0, \overline{\lambda}(x^*) \}$. Then there exists a neighbourhood $U$
of the point $x^*$ such that
$$
  \frac{f(x^*) - f(y)}{\varphi(y) - \varphi(x^*)} \le \lambda \quad \forall x \in U \cap (A \setminus M)
$$
or, equivalently,
\begin{equation} \label{PenFuncOutNearPoint}
  f(x^*) + \lambda \varphi(x^*) = F_{\lambda}(x^*) \le F_{\lambda}(y) = f(y) + \lambda \varphi(y) \quad
  \forall y \in U \cap (A \setminus M).
\end{equation}
Since $x^*$ is a locally optimal solution of the problem $(\mathcal{P})$, there exists a neighbourhood $V$ of $x^*$ such
that
\begin{equation} \label{PenFuncInNearPoint}
  f(x^*) \le f(y) \quad \forall y \in V \cap \Omega = V \cap (A \cap M).
\end{equation}
Define $U_0 = U \cap V$. The set $U_0$ is a neighbourhood of $x^*$. Taking into account (\ref{PenFuncOutNearPoint}),
(\ref{PenFuncInNearPoint}), and the fact that $\varphi(x) = 0$ for any $x \in M$ one gets that
$F_{\lambda}(x^*) \le F_{\lambda}(y)$  for all $y \in U_0 \cap A$. Therefore $x^*$ is a point of local minimum of
$F_{\lambda}$ on $A$, and $F_{\lambda}$ is exact at this point.

Note that from the proof of ``sufficiency'' part of the theorem it follows that 
$\lambda^*(x^*) \le \max\{ 0, \overline{\lambda}(x^*) \}$. Let us show that the equality 
$\lambda^*(x^*) = \max\{ \overline{\lambda}(x^*), 0 \}$ holds true. Indeed, if 
$\overline{\lambda}(x^*) \le 0$, then from the proof of the sufficiency part of the theorem it follows that for any
$\lambda > 0$ the point $x^*$ is a local minimizer of $F_{\lambda}$ on $A$, which yields the equality 
$\lambda^*(x^*) = 0$. 

Let now $\overline{\lambda}(x^*) > 0$, and fix an arbitrary $\lambda \in ( 0, \overline{\lambda}(x^*) )$. Then by 
the definition of limit superior one gets that for any neighbourhood $U$ of $x^*$ there 
exists $y \in U \cap (A \setminus M)$ such that
$$
  \frac{f(x^*) - f(y)}{\varphi(y) - \varphi(x^*)} > \lambda,
$$
which implies that $F_{\lambda}(x^*) > F_{\lambda}(y)$. Hence $x^*$ is not a point of local minimum of $F_{\lambda}$ on
$A$, and $\lambda^*(x^*) \ge \lambda$ for any $\lambda \in (0, \overline{\lambda}(x^*))$. Thus, 
$\lambda^*(x^*) = \overline{\lambda}(x^*)$, and the proof is complete.
\end{proof}

\subsection{Error Bounds and Local Exact Penalization}
\label{SubSecUpperEstim}

Theorem~\ref{ThNSCLocMin} provides a complete characterization of the exact penalty property in terms of the quantity
$\overline{\lambda}(x, f, \varphi)$; however, a direct computation of $\overline{\lambda}(x, f, \varphi)$ is very
difficult in nontrivial cases. In this subsection, we discuss a general technique for obtaining upper estimates of 
$\overline{\lambda}(x, f, \varphi)$ via error bounds, i.e. local estimates of the functions $f$ and $\varphi$ based on
the distance function $d(x, \Omega)$. This technique allows one to obtains many well--known results on exact penalty
functions as simple corollaries to Theorem~\ref{ThNSCLocMin}.

Let $(X, d)$ be a metric space. For any nonempty set $C \subset X$ and for all $x \in X$ 
denote by $d(x, C) = \inf_{y \in C} d(x, y)$ the distance between $x$ and $C$. Denote by 
$B(x, r) = \{ y \in X \mid d(x, y) \le r \}$ the closed ball with centre $x$ and radius $r > 0$, and by 
$U(x, r) = \{ y \in X \mid d(x, y) < r \}$ the open ball with the same centre and radius.

\begin{theorem} \label{ThUpEstimPenPar}
Let $(X, d)$ be a metric space, and $x^*$ be a locally optimal solution of the problem 
$(\mathcal{P})$. Suppose that there exist $r > 0$ and functions 
$\omega, \eta \colon \mathbb{R}_+ \to \mathbb{R}_+$ satisfying the following conditions:
\begin{enumerate}
\item{$\omega(0) = 0$, $\eta(0) = 0$ and $\eta(t) > 0$ for all $t > 0$;
}

\item{$f(y) \ge f(x^*) - \omega( d(y, \Omega) )$ for any $y \in B(x^*, r) \cap (A \setminus M)$;
}

\item{$\varphi(x) \ge \eta( d(x, \Omega) )$ for any $x \in B(x^*, r) \cap A$;
\label{ErrorBound}}

\item{$\limsup_{t \to +0} \frac{\omega(t)}{\eta(t)} =: \sigma < +\infty$.
}
\end{enumerate}
Then the penalty function $F_{\lambda}$ is exact at $x^*$, and $\lambda^*(x^*) \le \sigma$.
\end{theorem}

\begin{proof}
From the assumptions of the theorem it follows that for any $y \in B(x^*, r) \cap (A \setminus M)$ such that 
$d(y, \Omega) > 0$ one has
\begin{equation} \label{UpperEstim}
  \frac{f(x^*) - f(y)}{\varphi(y) - \varphi(x^*)} \le \frac{\omega( d(y, \Omega) )}{\eta(d(y, \Omega))}.
\end{equation}
On the other hand, if $y \in B(x^*, r) \cap (A \setminus M)$ and $d(y, \Omega) = 0$, then 
$f(x^*) - f(y) \le \omega( d(y, \Omega) ) = \omega(0) = 0$. Thus,
$$
  \overline{\lambda}(x^*) = 
  \limsup_{y \to x^*, y \in A \setminus M} \frac{f(x^*) - f(y)}{\varphi(y) - \varphi(x^*)} \le
  \limsup_{y \to x^*, y \in A \setminus M} h(y),
$$
where $h(y) = \omega( d(y, \Omega) ) / \eta( d(y, \Omega) )$, if $d(y, \Omega) > 0$, and $h(y) = 0$, otherwise.
Hence, as it is easy, one has
$$
  \overline{\lambda}(x^*) \le \limsup_{y \to x, y \in A \setminus M} h(y) \le 
  \limsup_{t \to +0} \frac{\omega(t)}{\eta(t)} = \sigma.
$$
It remains to apply Theorem~\ref{ThNSCLocMin}.  
\end{proof}

\begin{remark}
{(i) Note that assumption~\ref{ErrorBound} in the theorem above is an assumption on the existence of a local nonlinear
error bound for the penalty term $\varphi$. Therefore Theorem~\ref{ThUpEstimPenPar} establishes a connection between
Theorem~\ref{ThNSCLocMin} and many classical results on exact penalty functions well--known in the literature. In
particular, it allows one to use various results on error bounds 
\cite{Pang, WuYe, NgZheng, Zalinescu, Bosch, PenotChapter, BednKruger}, metric regularity and metric
subregularity \cite{Ioffe, Aze, LiMordukhovich, Kruger}, as well as results on subanalyticity of
functions and sets (see, for instance,~\cite{Dedieu, LuoPangRalphWu, LuoPangRalph, LinFukushima}) in order to verify the
exactness of linear penalty functions.
}

\noindent{(ii) One can obtain a similar result on the local exactness of the penalty function $F_{\lambda}$ under
slightly different assumptions. Namely, let, for the sake of simplicity, $A = X$, $Y$ be a normed space, 
$F \colon X \to Y$ be a given mapping, and $\Omega = M = \{ x \in X \mid F(x) = 0 \}$. Then instead of utilizing error
bounds one can directly suppose that in a neighbourhood of a given locally optimal solution $x^*$ the following
inequalities hold true:
$$
  f(x) \ge f(x^*) - \omega(\|F(x)\|), \quad \varphi(x) \ge \eta(\| F(x) \|).
$$
With the use of these inequalities one can easily prove the similar estimate 
$\lambda^*(x^*) \le \limsup_{t \to +0} \omega(t) / \eta(t)$. See \cite{HanMangasarian, HuyerNeumaier, WangMaZhou} for
the applications of this approach to some particular problems.
}

\noindent{(iii) It is worth mentioning that the new class of smooth exact penalty functions \cite{HuyerNeumaier,
WangMaZhou} that attracted a lot of attention of researchers recently, and found applications in various fields of
optimization and optimal control \cite{LiYu, JianLin, LinLoxton}, can be also studied by the methods developed in the
present article. In particular, one can strengthen existing results on the local exactness of penalty functions from
this class with the use of Theorem~\ref{ThUpEstimPenPar}.
}
\end{remark}

Let us show that the theorem above cannot be significantly sharpened. It is clear that if a locally optimal solution 
$x^* \in \Omega$ of the problem $(\mathcal{P})$ is a point of local minimum of the function $f$, then for any
nonnegative function $\varphi$ such that $M = \{ x \in X \mid \varphi(x) = 0 \}$ the penalty function $F_{\lambda}$ is
exact at $x^*$. Therefore, in order to prove that the conditions of Theorem \ref{ThUpEstimPenPar} are nearly optimal one
should study whether the penalty function $F_{\lambda}$ is exact at a given point for some classes of functions $f$ and
$\varphi$.

\begin{proposition}
Let $(X, d)$ be a metric space, $\Omega$ be a closed set, and $x^*$ be a locally optimal solution of the problem
$(\mathcal{P})$. Suppose also that $\eta \colon \mathbb{R}_+ \to \mathbb{R}_+$ is a function
such that $\eta(t) = 0$ iff $t = 0$. If for any nonnegative function $\varphi$ such that
\begin{enumerate}
\item{$M = \{ x \in X \mid \varphi(x) = 0 \}$,}

\item{for some $r > 0$ one has $\varphi(x) \ge \eta(d(x, \Omega))$ for all $x \in B(x^*, r) \cap A$}
\end{enumerate}
the penalty function $F_{\lambda} = f + \lambda \varphi$ is exact at $x^*$, then
\begin{equation} \label{LowEstimF}
  f(y) \ge f(x^*) - L \eta( d(y, \Omega) ) \quad \forall y \in B(x^*, r) \cap (A \setminus M)
\end{equation}
for some $L > 0$ and $r > 0$.
\end{proposition}

\begin{proof}
Set $\varphi(\cdot) = \eta(d(\cdot, \Omega))$. From the fact that the penalty function $F_{\lambda}$ is exact at $x^*$,
and Theorem~\ref{ThNSCLocMin} it follows that for any $L > \max\{ \overline{\lambda}(x^*), 0 \}$ there exists $r > 0$
such that
$$
  \frac{f(x^*) - f(y)}{\varphi(y) - \varphi(x^*)} \le L \quad \forall y \in B(x^*, r) \cap (A \setminus M),
$$
which implies inequality~(\ref{LowEstimF}).  
\end{proof}

\begin{proposition} \label{PrpLowEstimOfPhi}
Let $(X, d)$ be a metric space, $x^* \in \Omega$, and $\omega \colon \mathbb{R}_+ \to \mathbb{R}_+$ be a given
function such that $\omega(0) = 0$. If for any function $f \colon X \to \mathbb{R} \cup \{ +\infty \}$ such that
\begin{enumerate}
\item{$x^*$ is a point of local minimum of $f$ on $\Omega$,}

\item{$f(y) \ge f(x^*) - \omega( d(y, \Omega) )$ for all $y \in B(x^*, r) \cap (A \setminus M)$ for some $r > 0$}
\end{enumerate}
the penalty function $F_{\lambda} = f + \lambda \varphi$ is exact at $x^*$, then
$$
  \varphi(x) \ge a \omega( d(x, \Omega) ) \quad \forall x \in B(x^*, r) \cap A
$$
for some $a > 0$ and $r > 0$.
\end{proposition}

\begin{proof}
Define $f(\cdot) = -\omega(d(\cdot, \Omega))$. Then $x^*$ is a point of global minimum of $f$ on $\Omega$ due to the
fact that $f(x) = 0$ for any $x \in \Omega$. Applying the fact that the penalty function $F_{\lambda}$ is exact at
$x^*$, and Theorem~\ref{ThNSCLocMin} one gets that for any $L > \max\{ \overline{\lambda}(x^*), 0 \}$ there 
exists $r > 0$ such that
$$
  \frac{f(x^*) - f(y)}{\varphi(y) - \varphi(x^*)} \le L \quad \forall y \in B(x^*, r) \cap (A \setminus M).
$$
Consequently,
$$
  \varphi(x) \ge \frac{1}{L} \omega( d(x, \Omega) ) \quad \forall x \in B(x^*, r) \cap (A \setminus M).
$$
It remains to define $a = 1 / L$, and note that if $x \in B(x^*, r) \cap (A \cap M)$, then $x \in \Omega$ and
$\varphi(x) = 0 = \omega(0) = \omega( d(x, \Omega) )$.  
\end{proof}

In the propositions above the function $f$ is assumed to satisfy the inequality
$$
  f(y) \ge f(x^*) - \omega( d(y, \Omega) ) \quad \forall y \in B(x^*, r) \cap (A \setminus M)
$$
for a nonnegative function $\omega$. Let us show that this inequality follows from some more widely used conditions,
namely Lipschitz and H\"older continuity.

\begin{proposition} \label{PrpLoweEstimOfF}
Let $(X, d)$ be a metric space, and $x^*$ be a locally optimal solution of the problem $(\mathcal{P})$. Suppose that
there exist $r > 0$, and a continuous from the right function $\omega \colon \mathbb{R}_+ \to \mathbb{R}_+$ such that
$$
  f(x) - f(y) \le \omega( d(x, y) ) \quad \forall x \in B(x^*, r) \cap \Omega \quad
  \forall y \in B(x^*, r) \cap (A \setminus M).
$$
Then there exists $\delta > 0$ such that
$$
  f(y) \ge f(x^*) - \omega( d(y, \Omega) ) \quad \forall y \in B(x^*, \delta) \cap (A \setminus M).
$$
\end{proposition}

\begin{proof}
By the definition of locally optimal solution there exists $r_0 > 0$ such that
$$
  f(x) \ge f(x^*) \quad \forall x \in B(x^*, r_0) \cap \Omega = B(x^*, r_0) \cap (M \cap A).
$$
Denote $\delta = \min\{ r, r_0 \} / 2$, and fix an arbitrary $y \in B(x^*, \delta) \cap (A \setminus M)$. By definition
there exists a sequence $\{ x_n \} \subset \Omega$ such that for all $n \in \mathbb{N}$
\begin{equation} \label{DecrConvToDist}
  d(y, x_n) \xrightarrow[n \to \infty]{} d(y, \Omega), \quad d(y, x_n) \le d(y, x^*) \le \frac{r_0}{2}, \quad
  d(y, x_n) \ge d(y, x_{n + 1}).
\end{equation}
Therefore 
$$
  d(x^*, x_n) \le d(x^*, y) + d(y, x_n) \le \frac{r_0}{2} + \frac{r_0}{2} = r_0.
$$
Thus, $\{ x_n \} \subset B(x^*, r_0) \cap \Omega$, which implies that $f(x^*) \le f(x_n)$ for all $n \in \mathbb{N}$.
Hence for any $n \in \mathbb{N}$ one has
$$
  f(x^*) - f(y) = f(x^*) - f(x_n) + f(x_n) - f(y) \le f(x_n) - f(y) \le \omega( d(x_n, y) ).
$$
Taking into account (\ref{DecrConvToDist}), and passing to the limit as $n \to \infty$ one gets the desired result.  
\end{proof}

\begin{corollary}
Let $(X, d)$ be a metric space, and $x^*$ be a locally optimal solution of the problem $(\mathcal{P})$. Suppose that $f$
is H\"older continuous with exponent $\alpha > 0$ in a neighbourhood of $x^*$. Then there exist $C > 0$ and $\delta > 0$
such that
$$
  f(y) \ge f(x^*) - C (d(y, \Omega))^\alpha \quad \forall y \in B(x^*, \delta) \cap (A \setminus M).
$$
\end{corollary}

As another simple corollary to Theorem~\ref{ThUpEstimPenPar}, and Propositions~\ref{PrpLowEstimOfPhi} and
\ref{PrpLoweEstimOfF} one gets the following result.

\begin{corollary}
Let $(X, d)$ be a metric space, and $x^*$ be a locally optimal solution of the problem $(\mathcal{P})$. Then
the penalty function $F_{\lambda} = f + \lambda \varphi$ is exact at $x^*$ for any function $f$ that is H\"older
continuous with exponent $\alpha \in (0, 1]$ near $x^*$ if and only if there exist $r > 0$ and $a > 0$ such that
$$
  \varphi(x) \ge a (d(x, \Omega))^{\alpha} \quad \forall x \in B(x^*, r) \cap A.
$$
\end{corollary}

Let us mention some other necessary conditions for a penalty function to be locally exact that can be directly deduced
from Theorem~\ref{ThNSCLocMin}.

\begin{proposition}
Let $(X, d)$ be a metric space, $x^*$ be a locally optimal solution of the problem $(\mathcal{P})$, and $F_{\lambda}$ be
exact at $x^*$. Suppose that there exist $r > 0$, and a function $\omega \colon \mathbb{R}_+ \to \mathbb{R}_+$ such that
$$
  \varphi(x) - \varphi(x^*) \le \omega(d(x, x^*)) \quad \forall x \in B(x^*, r).
$$
Then there exist $\delta > 0$ and $L > 0$ such that
$$
  f(x) - f(x^*) \ge - L \omega(d(x, x^*)) \quad \forall x \in B(x^*, \delta) \cap (A \setminus M).
$$
In particular, in the case $\omega(t) \equiv t$ one has that the Lipschitz continuity of $\varphi$ near $x^*$ implies
the calmness from below of $f$ at $x^*$.
\end{proposition}

\begin{proposition}
Let $(X, d)$ be a metric space, and $x^*$ be a locally optimal solution of the problem $(\mathcal{P})$. Suppose that
there exist $r > 0$, and a function $\omega \colon \mathbb{R}_+ \to \mathbb{R}_+$ such that
$$
  f(x) \le f(x^*) - \omega(d(x, x^*)) \quad \forall x \in B(x^*, r) \cap (A \setminus M).
$$
Then for $F_{\lambda}$ to be exact at $x^*$ it is necessary that there exist $\delta > 0$ and $a > 0$ such that
$$
  \varphi(x) - \varphi(x^*) \ge a \omega(d(x, x^*)) \quad \forall x \in B(x^*, \delta) \cap (A \setminus M).
$$
\end{proposition}

In the end of this subsection, we consider one simple example that illustrates the main results discussed above.

\begin{example}
Let $X = A = \mathbb{R}$, and $M = \Omega = (- \infty, 0]$. Define
$$
  f(x) = \begin{cases}
    -x, \text{ if } x \le 0, \\
    -(x + 1)^2 + 1, \text{ if } x > 0,
  \end{cases}
$$
and set $\varphi(x) = \max\{ 0, x \}$. Obviously, $x^* = 0$ is a point of global minimum of $f$ on $\Omega$. Note also
that both function $f$ and $\varphi$ are locally Lipschitz continuous.

It is clear that $\varphi(x) = d(x, \Omega)$, and that for any $r > 0$ the function $f$ is Lipschitz continuous on 
$(-r, r)$ with a Lipschitz constant $L = 2(r + 1)$. Therefore by Proposition~\ref{PrpLoweEstimOfF} and
Theorem~\ref{ThUpEstimPenPar} the penalty function $F_{\lambda}$ is exact at the point $x^* = 0$, and
$\overline{\lambda}(x^*) \le 2(r + 1)$ for any $r > 0$, which implies $\overline{\lambda}(x^*) \le 2$. A direct
computation shows that $\overline{\lambda}(x^*) = 2$. Hence the least exact penalty parameter $\lambda^*(x^*)$ at the
point $x^* = 0$ equals $2$.

Note that $F_2(x) = -x^2$ for any $x \ge 0$. Therefore $x^* = 0$ is not a point local minimum of $F_{\lambda}$ with 
$\lambda = 2$. 
\end{example}

\begin{remark} \label{RmrkLocalExPenPar}
The example above demonstrates that, in the general case, a locally optimal solution of the problem $(\mathcal{P})$ is
not necessarily a point of local minimum of the penalty function $F_{\lambda}$ on the set $A$ when 
$\lambda = \lambda^*(x^*)$.
\end{remark}

\subsection{Problem Calmness and Local Exact Penalization}
\label{SectProbCalmness}

A different approach to the study of local exactness is based on an analysis of an optimization problem
behaviour under perturbations of constraints. This approach allows one to avoid any direct usage of error bounds
and Lipschitz-like behaviour of the objective function and, thus, can be applied to a broader class of
optimization problems than the technique discussed in the previous subsection. 

The approach discussed in this subsection was originally proposed by Rockafellar and Clarke \cite{Clarke}, and
later on was developed by many different authors \cite{ClarkeBook, BurkeCalm, RalphYang, HuangTeoYang, ZhaiHuang} (see
also \cite{Wen, Henrion, Penot}). Our exposition of the subject is a straightforward generalization of the ideas
developed in \cite{UderzoCalm}, where the optimization problem calmness under nonlinear perturbations was studied. Note
that the results of this subsection also generalize the concept of lower order calmness that was studied in
\cite{LowerOrderCalmness}.

Consider the perturbed family of constrained optimization problems
$$
  \min f(x) \quad \text{subject to} \quad x \in M(p), \quad x \in A, \eqno (\mathcal{P}_p)
$$
where $M \colon P \rightrightarrows X$ is a given set--valued mapping, $(P, d)$ is a metric space of perturbation
parameters, and $M(p^*) = M$ for some $p^* \in P$. Thus, the problem ($\mathcal{P}_p$) coincides with the problem
$(\mathcal{P})$ when $p = p^*$. Denote by $\Omega(p) = M(p) \cap A$ the set of feasible points
of the problem ($\mathcal{P}_p$). Recall that $\Omega^{-1} \colon X \to P$ with
$\Omega^{-1} (x) = \{ p \in P \mid x \in \Omega(p) \}$ for any $x \in X$ is the inverse multifunction to $\Omega$.

\begin{definition}
Let $x^*$ be a locally optimal solution of the problem $(\mathcal{P})$, and 
$\omega \colon \mathbb{R}_+ \to \mathbb{R}_+$ be a given function. The problem $(\mathcal{P}_{p^*})$ is called 
$\omega$-\textit{calm} at $x^*$ if there exist $r > 0$, $a > 0$, and a neighbourhood $U$ of $x^*$ such that
$$
  f(x) \ge f(x^*) - a \omega(d(p, p^*)) \quad \forall x \in U \cap \Omega(p) \quad \forall p \in B(p^*, r).
$$
\end{definition}

Utilizing the notion of $\omega$-calmness we can give a sufficient condition for the penalty function $F_{\lambda}$
to be exact at a locally optimal solution of the problem $(\mathcal{P})$. We use general Theorem~\ref{ThNSCLocMin} in
order to obtain this result.

\begin{theorem} \label{ThCalmLocalPen}
Let $X$ be a topological space, $x^*$ be a locally optimal solution of the problem $(\mathcal{P})$, 
$\omega, \eta \colon \mathbb{R}_+ \to \mathbb{R}_+$ be given functions such that
$\eta(t) = 0$ iff $t = 0$, and $\eta$ is strictly increasing and continuous from the right. Suppose that
\begin{enumerate}
\item{$f$ is lower semicontinuous (l.s.c.) at $x^*$;}

\item{the problem $(\mathcal{P}_{p^*})$ is $\omega$-calm at $x^*$;}

\item{$\varphi(x) \ge \eta( d(p^*, \Omega^{-1}(x)) )$ for any $x \in V$,
where $V$ is a neighbourhood of $x^*$;}

\item{$\limsup_{t \to +0} \frac{\omega(t)}{\eta(t)} =: \sigma < +\infty$.}
\end{enumerate}
Then the penalty function $F_{\lambda}$ is exact at $x^*$ and $\lambda^*(x^*) \le a \sigma$, where $a > 0$ is
from the definition of $\omega$-calmness.
\end{theorem}

\begin{proof}
Let us verify that $\overline{\lambda}(x^*) \le a \sigma$. Then applying Theorem~\ref{ThNSCLocMin} one gets the desired
result. Let us also note that one can obviously suppose that $x^*$ is a limit point of the set $A \setminus M$.

Fix an arbitrary $\varepsilon > 0$, and choose a net $\{ x_{\gamma} \} \subset A \setminus M$, $\gamma \in \Gamma$,
converging to $x^*$. By the definition of $\omega$-calmness there exist $r > 0$, $a > 0$, and a neighbourhood $U$ of
$x^*$ such that
$$
  f(x) \ge f(x^*) - a \omega(d(p, p^*)) \quad \forall x \in U \cap \Omega(p) \quad \forall p \in B(p^*, r),
$$
while by the definition of limit superior there exists $t_0 > 0$ such that $\omega(t) / \eta(t) < \sigma + \varepsilon$
for any $t \in (0, t_0)$. Define the set
$$
  \Gamma_1 = \big\{ \gamma \in \Gamma \mid \varphi(x_{\gamma}) \ge \min\{ \eta(r), \eta(t_0) \} \big\},
$$
and denote $\Gamma_2 = \Gamma \setminus \Gamma_1$.

The function $f$ is l.s.c. at $x^*$. Therefore there exists $\gamma_1 \in \Gamma$ such that  
$$
  f(x_{\gamma}) \ge f(x^*) - \varepsilon \min\{ \eta(r), \eta(t_0) \} \quad \forall \gamma \ge \gamma_1.
$$
Hence for any $\gamma \in \Gamma_1$ such that $\gamma \ge \gamma_1$ (note that there might be no such $\gamma$) one has
\begin{equation} \label{LSCofFandDefOfGamma}
  \frac{f(x^*) - f(x_{\gamma})}{\varphi(x_{\gamma}) - \varphi(x^*)} \le 
  \frac{\varepsilon \min\{ \eta(r), \eta(t_0) \}}{\min\{ \eta(r), \eta(t_0) \}} = \varepsilon.
\end{equation}

From the fact that the net $\{ x_{\gamma} \}$, $\gamma \in \Gamma$, converges to $x^*$ it follows that there exists
$\gamma_2 \in \Gamma$ such that $x_{\gamma} \in U \cap V$ for any $\gamma \ge \gamma_2$, where $U$ is from the
definition of $\omega$-calmness, and $V$ is from the third condition of the theorem. Hence for any $\gamma \in \Gamma_2$
such that $\gamma \ge \gamma_2$ one has
\begin{equation} \label{ParamErrorBoundPenaltyTerm}
  \eta\big( d(p^*, \Omega^{-1}(x_{\gamma})) \big) \le \varphi(x_{\gamma}) < \min\{ \eta(r), \eta(t_0) \},
\end{equation}
which implies $d(p^*, \Omega^{-1}(x_{\gamma})) < \min\{ r, t_0 \}$, since $\eta$ is strictly increasing. 

Fix an arbitrary $\gamma \in \Gamma_2$ such that $\gamma \ge \gamma_2$, and denote 
$\tau = d(p^*, \Omega^{-1}(x_{\gamma}))$. Observe that $\varphi(x_{\gamma}) > 0$ and 
$p^* \notin \Omega^{-1}(x_{\gamma})$ for any $\gamma \in \Gamma$ due to the fact that $x_{\gamma} \in A \setminus M$ by
definition. Applying the continuity from the right of the function $\eta$ one obtains that 
there exists $\Delta \tau > 0$ such that
$$
  \eta(\tau) \ge \eta(t) - \varepsilon \varphi(x_{\gamma}) \quad \forall t \in [\tau, \tau + \Delta \tau).
$$
Hence and from (\ref{ParamErrorBoundPenaltyTerm}) it follows that
$$
  \eta(t) \le (1 + \varepsilon) \varphi(x_{\gamma}) \quad \forall t \in [\tau, \tau + \Delta \tau).
$$
Recall that $\tau = d(p^*, \Omega^{-1}(x_{\gamma})) < \min\{ r, t_0 \}$. Therefore there exists
$p_{ \gamma } \in \Omega^{-1}(x_{\gamma})$ such that 
$\tau \le d(p^*, p_{\gamma}) < \min \{ r, t_0, \tau + \Delta \tau \}$, which yields
\begin{equation} \label{ChoiceOfpGamma}
  \eta( d(p^*, p_{\gamma}) ) \le (1 + \varepsilon) \varphi(x_{\gamma}), \quad 
  d(p^*, p_{\gamma}) < \min\{ r, t_0 \}
\end{equation}
for any $\gamma \in \Gamma_2$ such that $\gamma \ge \gamma_2$. As it was noted above, 
$p^* \notin \Omega^{-1}(x_{\gamma})$. Consequently, $d(p^*, p_{\gamma}) > 0$ due to the fact that 
$p_{\gamma} \in \Omega^{-1}(x_{\gamma})$. Therefore applying (\ref{ChoiceOfpGamma}) and $\omega$-calmness of the
problem $(\mathcal{P}_{p^*})$ one gets that for any $\gamma \in \Gamma_2$ such that $\gamma \ge \gamma_2$
\begin{equation} \label{UpperEstViaPertParam}
  \frac{f(x^*) - f(x_{\gamma})}{\varphi(x_{\gamma}) - \varphi(x^*)} \le
  \frac{a(1 + \varepsilon) \omega(d(p^*, p_{\gamma}))}{\eta( d(p^*, p_{\gamma}) )} \le
  a (1 + \varepsilon) (\sigma + \varepsilon)
\end{equation}
by the fact that $d(p^*, p_{\gamma}) < \min\{ r, t_0 \}$ and the choice of $t_0$.

By the definition of net there exists $\gamma_0 \in \Gamma$ such that $\gamma_0 \ge \gamma_1$ and 
$\gamma_0 \ge \gamma_2$. Taking into account (\ref{LSCofFandDefOfGamma}) and
(\ref{UpperEstViaPertParam}) one gets that for any $\gamma \in \Gamma$ such that $\gamma \ge \gamma_0$
$$
  \frac{f(x^*) - f(x_{\gamma})}{\varphi(x_{\gamma}) - \varphi(x^*)} \le
  \max\big\{ \varepsilon, a (1 + \varepsilon) (\sigma + \varepsilon) \big\},
$$
which implies
$$
  \limsup_{\gamma \in \Gamma} \frac{f(x^*) - f(x_{\gamma})}{\varphi(x_{\gamma}) - \varphi(x^*)} \le a \sigma.
$$
Hence $\overline{\lambda}(x^*) \le a\sigma$, since the net $\{ x_{\gamma} \} \subset A \setminus M$, 
$\gamma \in \Gamma$, was chosen arbitrarily.  
\end{proof}

Let us show that Theorem~\ref{ThCalmLocalPen} is, in essence, a parametric counterpart of Theorem~\ref{ThUpEstimPenPar}
in which the direct estimates of the form
$$
  \varphi(x) \ge \eta(d(x, \Omega)), \quad f(x) \ge f(x^*) - \omega(d(x, \Omega))
$$
are replaced by the indirect estimates (i.e. the estimates obtained via perturbation)
$$
  \varphi(x) \ge \eta(d(p^*, \Omega^{-1}(x)), \quad f(x) \ge f(x^*) - \omega(d(p^*, \Omega^{-1}(x)),
$$
that might exist in a more general case. The theorem below provides an equivalent formulation of the $\omega$-calmness
of a perturbed optimization problem, and contains some existing results (see~\cite{BurkeCalm, LowerOrderCalmness}) as
simple particular cases.

\begin{theorem} 
Let $x^*$ be a locally optimal solution of the problem $(\mathcal{P})$, and 
$\omega \colon \mathbb{R}_+ \to \mathbb{R}_+$ be a strictly increasing and continuous from the right function such that
$\omega(0) = 0$. Suppose also that $f$ is l.s.c. at $x^*$. Then the problem $(\mathcal{P}_{p^*})$ is $\omega$-calm at
$x^*$ if and only if there exists a neighbourhood $U$ of $x^*$ such that
\begin{equation} \label{IndirectLowerEstim_ObjFunc}
  f(x) \ge f(x^*) - a \omega(d(p^*, \Omega^{-1}(x))) \quad \forall x \in U,
\end{equation}
where $a > 0$ is from the definition of the $\omega$-calmness of the problem $(\mathcal{P}_{p^*})$ at $x^*$.
\end{theorem}

\begin{proof}
Let there exist a neighbourhood $U$ of $x^*$ such (\ref{IndirectLowerEstim_ObjFunc}) holds true. Then for any 
$p \in P$ and $x \in U \cap \Omega(p)$ one has
$$
  f(x) \ge f(x^*) - a \omega(d(p^*, \Omega^{-1}(x))) \ge f(x^*) - a \omega(d(p^*, p))
$$
due to the facts that the function $\omega$ is strictly increasing, and $x \in \Omega(p)$ iff $p \in \Omega^{-1}(x)$.
Thus, the problem $(\mathcal{P}_{p^*})$ is $\omega$-calm at $x^*$ with $r = +\infty$.

Suppose now that the problem $(\mathcal{P}_{p^*})$ is $\omega$-calm at $x^*$. Then there exist $r > 0$, $a > 0$, and a
neighbourhood $U$ of $x^*$ such that 
\begin{equation} \label{OmegaCalmOfPertProb}
  f(x) \ge f(x^*) - a \omega(d(p, p^*)) \quad \forall x \in U \cap \Omega(p) \quad \forall p \in B(p^*, r).
\end{equation}
Choose a decreasing sequence $\{ \varepsilon_n \}$ such that $\varepsilon_n > 0$ for any $n \in \mathbb{N}$, and
$\varepsilon_n \to 0$ as $n \to \infty$. Applying the lower semicontinuity of the function $f$ at $x^*$ one obtains
that for any $n \in \mathbb{N}$ there exists a neighbourhood $U_n$ of $x^*$ such that
$f(x) \ge f(x^*) - \varepsilon_n$ for all $x \in U_n$.

Arguing by reductio ad absurdum, suppose that (\ref{IndirectLowerEstim_ObjFunc}) does not hold true. Then, in
particular, for any $n \in \mathbb{N}$ there exists $x_n \in U_n \cap U$ such that
\begin{equation} \label{AbsordoParamEstim}
  f(x_n) < f(x^*) - a \omega( d(p^*, \Omega^{-1}(x_n) )
\end{equation}
Note that for all $n \in \mathbb{N}$ one has $f(x_n) \ge f(x^*) - \varepsilon_n$ by the definition of $U_n$. Therefore
$\omega( d(p^*, \Omega^{-1}(x_n) ) < \varepsilon_n / a$ for any $n \in \mathbb{N}$, which implies that 
$d(p^*, \Omega^{-1}(x_n)) \to 0$ as $n \to \infty$ due to the fact that the function $\omega$ is strictly increasing,
and $\omega(0) = 0$. Hence for any sufficiently large $n$ one has 
$d(p^*, \Omega^{-1}(x_n)) < r / 2$. Consequently, applying (\ref{AbsordoParamEstim}) and the continuity from the right
of the function $\omega$ one gets that for any $n$ large enough there exists $p_n \in \Omega^{-1}(x_n)$ such that 
$d(p^*, \Omega^{-1}(x_n)) \le d(p^*, p_n) < r$ and
$$
  f(x_n) < f(x^*) - a \omega( d(p^*, p_n) ),
$$
which contradicts (\ref{OmegaCalmOfPertProb}) by virtue of the facts that $x_n \in U$ by construction, and 
$p_n \in \Omega^{-1}(x_n)$ iff $x_n \in \Omega(p_n)$.
\end{proof}

Imposing different assumptions on the penalty term $\varphi$ than in Theorem~\ref{ThCalmLocalPen} one can prove that the
$\omega$-calmness of the problem $(\mathcal{P}_{p^*})$ is also necessary for the penalty function to be exact at a
locally optimal solution of this problem.

\begin{proposition}
Let $X$ be a topological space, and $x^*$ be a locally optimal solution of the problem $(\mathcal{P})$. Suppose that
there exist $r > 0$, a neighbourhood $U$ of $x^*$, and a function $\eta \colon \mathbb{R}_+ \to \mathbb{R}_+$ such that
$$
  \varphi(x) \le \eta(d(p, p^*)) \quad \forall x \in U \cap \Omega(p) \quad 
  \forall p \in B(p^*, r).
$$
If the penalty function $F_{\lambda}$ is exact at $x^*$, then the problem $(\mathcal{P}_{p^*})$ is $\omega$-calm at
$x^*$ for any function $\omega$ such that $\omega(\cdot) \ge \eta(\cdot)$.
\end{proposition} 

\begin{proof}
Let $F_{\lambda}$ be exact at $x^*$. Fix an arbitrary $\lambda > \lambda^*(x^*)$. Then there exists a neighbourhood
$V$ of $x^*$ such that
$$
  f(x) + \lambda \varphi(x) = F_{\lambda}(x) \ge F_{\lambda}(x^*) = f(x^*) \quad \forall x \in V \cap A.
$$
Therefore
$$
  f(x) - f(x^*) \ge - \lambda \varphi(x) \ge - \lambda \eta(d(p, p^*)) \quad 
  \forall x \in U_0 \cap \Omega(p) \quad \forall p \in B(p^*, r), 
$$
where $U_0 = U \cap V$. Thus, the problem $(\mathcal{P}_{p^*})$ is $\omega$-calm at $x^*$ for any function $\omega$ such
that $\omega(\cdot) \ge \eta(\cdot)$.  
\end{proof}

\begin{remark} \label{RmrkLinDependOnParam}
Let $f$ be l.s.c., $Y$ be a normed space, $P = Y$, and 
$$
  \Omega = \{ x \in X \mid 0 \in \Phi(x) \}, \quad \Omega(p) = \{ x \in X \mid 0 \in \Phi(x) - p \},
$$
where $\Phi \colon X \rightrightarrows Y$ is a set-valued mapping with closed values 
(cf.~\cite{UderzoCalm}, Corollary~3.1). Define $\varphi(x) = d(0, \Phi(x))$ and $p^* = 0$. 
Then $\Omega^{-1}(x) = \Phi(x)$, and, as it is easy to verify, one has
$$
  \varphi(x) = d(p^*, \Omega^{-1}(x)) \quad \forall x \in X, \quad
  \varphi(x) \le d(p, p^*) \quad \forall x \in \Omega(p) \quad \forall p \in P.
$$
Therefore by Theorem~\ref{ThCalmLocalPen} and the proposition above the penalty function $F_{\lambda}$ is exact at a
locally optimal solution of the problem $(\mathcal{P})$ if and only if the problem $(\mathcal{P}_{p^*})$ is
calm (i.e. $\omega$-calm for $\omega(t) \equiv t$) at this point. Roughly speaking, if a perturbation of a problem is
``linear'', then the calmness of this problem at a given point is equivalent to the exactness of the penalty function at
this point.
\end{remark}

\begin{remark} \label{RmrkParamViaPhi}
Note that the problem $(\mathcal{P})$ is equivalent to the problem
$$
  \min f(x) \quad \text{subject to} \quad \varphi(x) \le 0, \quad x \in A.
$$
Hence one can consider the following perturbation of the initial problem
\begin{equation} \label{PerProbInTermsOfPhi}
  \min f(x) \quad \text{subject to} \quad \varphi(x) \le p, \quad x \in A,
\end{equation}
where $p \ge 0$ and $p^* = 0$. Therefore, as it is easy to see, if $f$ is l.s.c., then the penalty function
$F_{\lambda}$ is exact at a locally optimal solution $x^*$ of the problem $(\mathcal{P})$ iff the problem
(\ref{PerProbInTermsOfPhi}) with $p = 0$ is calm at $x^*$. Note also that the inequality constraint $\varphi(x) \le p$
can be replaced with the equality constraint $\varphi(x) = p$ for $p \ge 0$.
\end{remark}

\section{Global Theory of Exact Penalty Functions}
\label{SectGlobTheory}

In this section, we study (globally) exact penalty functions. A penalty function is called (globally) \textit{exact} if
any point of global minimum of this function is also a globally optimal solution of the initial constrained optimization
problem. We discuss several approaches to the study of global exactness of penalty functions, and obtain some necessary
and sufficient conditions for a penalty function to be exact under different assumptions on the space $X$, and 
the functions $f$ and $\varphi$.

Throughout this section, we assume that there exists a globally optimal solution of the problem $(\mathcal{P})$, i.e.
that $f$ attains a global minimum on $\Omega$.

\subsection{Some Properties of Penalty Functions}

In this subsection, we discuss some properties of global minimizers of penalty functions. Most of these properties
are well-known.

For the sake of convenience define a set-valued mapping $G \colon \mathbb{R}_+ \rightrightarrows A$, where
$$
  G(\lambda) = \argmin_{x \in A} F_{\lambda}(x), \quad \forall \lambda \ge 0,
$$
i.e. $G(\lambda)$ is the set of all global minimizers of $F_{\lambda}$ on the set $A$. By definition, if for some 
$\lambda \ge 0$ the penalty function $F_{\lambda}$ does not attain a global minimum on the set $A$, then 
$G(\lambda) = \emptyset$. Recall that $\dom G = \{ \lambda \ge 0 \mid G(\lambda) \ne \emptyset \}$.

\begin{proposition}\cite{Demyanov} \label{PrpPropertOfPenFunc}
For any $\mu > \lambda \ge 0$ such that $\mu, \lambda \in \dom G$, and for all $x_{\lambda} \in G(\lambda)$,
$x_{\mu} \in G(\mu)$ the following hold:
\begin{enumerate}
\item{$f(x_{\mu}) \ge f(x_{\lambda})$ and $\varphi(x_{\mu}) \le \varphi(x_{\lambda})$;}

\item{if $\varphi(x_{\lambda}) = 0$, then $\varphi(x_{\mu}) = 0$, and $x_{\lambda}$ (as well as $x_{\mu}$) is a globally
optimal solution of the problem $(\mathcal{P})$.
}
\end{enumerate}
\end{proposition}

\begin{proposition} \label{PrpDefOfExPenFunc}
Let $\mu > \lambda \ge 0$ be such that $\mu, \lambda \in \dom G$, and let $x_{\lambda} \in G(\lambda)$ and
$x_{\mu} \in G(\mu)$. Suppose that $F_{\lambda}(x_{\lambda}) = F_{\mu}(x_{\mu})$. Then for all $\nu > \lambda$ one has
\begin{equation} \label{EquivOfOptimProblems}
  G(\nu) := \argmin_{x \in A} F_{\nu}(x) = \argmin_{x \in \Omega} f(x)
\end{equation}
or, equivalently, the sets of globally optimal solutions of the problem $(\mathcal{P})$, and of the problem
$$
  \min F_{\nu}(x) \quad \text{subject to} \quad x \in A
$$
coincide.
\end{proposition}

\begin{proof}
Let us show, at first, that $\varphi(x_{\mu}) = 0$. Indeed, arguing by reductio ad absurdum, suppose that
$\varphi(x_{\mu}) > 0$. Then
$$
  F_{\mu}(x_{\mu}) = f(x_{\mu}) + \mu \varphi(x_{\mu}) = 
  F_{\lambda}(x_{\mu}) + (\mu - \lambda) \varphi(x_{\mu}) >
  F_{\lambda}(x_{\mu}),
$$
since $\mu > \lambda$, and $\varphi(x_{\mu}) > 0$. Consequently, 
$$
  F_{\mu}(x_{\mu}) > F_{\lambda}(x_{\mu}) \ge F_{\lambda}(x_{\lambda})
$$
by the definition of $x_{\lambda}$, which contradicts the assumptions. Thus, $\varphi(x_{\mu}) = 0$.

Fix an arbitrary $\nu > \lambda$. Taking into account the fact that $\varphi(x_{\mu}) = 0$ one obtains that
$$
  F_{\lambda}(x_{\lambda}) = F_{\mu}(x_{\mu}) = f(x_{\mu}) = f^* := \min_{x \in \Omega} f(x).
$$
Applying the fact that $F_{\lambda}$ is nondecreasing with respect to $\lambda$ one gets that
$$
  F_{\lambda}(x_{\lambda}) = \min_{x \in A} F_{\lambda}(x) \le \inf_{x \in A} F_{\nu}(x) \le
  \inf_{x \in \Omega} F_{\nu}(x) = \min_{x \in \Omega} f(x),
$$
which yields
$$
  \inf_{x \in A} F_{\nu}(x) = F_{\lambda}(x_{\lambda}) = f^*.
$$
Consequently, any globally optimal solution of the problem $(\mathcal{P})$ is a point of global minimum of $F_{\nu}$ on
the set $A$ or, equivalently, $\nu \in \dom G$ and $\argmin_{x \in \Omega} f(x) \subseteq G(\nu)$. On the other hand,
if $x_{\nu} \in G(\nu)$, then arguing in the same way as in the case of $x_{\mu}$ one can check that 
$\varphi(x_{\nu}) = 0$, which implies the desired result.  
\end{proof}

\begin{corollary} \label{CrlrDefExPenFunc}
Suppose that there exists $\lambda > 0$ such that
$$
  \inf_{x \in A} F_{\lambda}(x) = \min_{x \in \Omega} f(x) =: f^*
$$
(in particular, one can suppose that there exists $x_{\lambda} \in G(\lambda)$ such that 
$\varphi(x_{\lambda}) = 0$). Then for any $\mu > \lambda$ one has $G(\mu) = \argmin_{x \in \Omega} f(x)$.
\end{corollary}

Proposition~\ref{PrpDefOfExPenFunc} and the corollary above motivate us to introduce the definition of exact penalty
function.

\begin{definition}
The penalty function $F_{\lambda} = f + \lambda \varphi$ is called \textit{exact} (for the problem $(\mathcal{P})$) if
there exists $\lambda_0 \ge 0$ such that for any $\lambda > \lambda_0$ the set of global minimizers of $F_{\lambda}$ on
$A$ coincides with the set of globally optimal solutions of the problem $(\mathcal{P})$. The greatest lower bound of all
such $\lambda_0 \ge 0$ is referred to as \textit{the least exact penalty parameter} of the penalty function
$F_{\lambda}$, and is denoted by $\lambda^*(f, \varphi)$.
\end{definition}

\begin{remark}
Note that from Corollary~\ref{CrlrDefExPenFunc} it follows that the penalty function $F_{\lambda}$ is exact iff 
the equality $\inf_{x \in A} F_{\lambda}(x) = f^*$ holds true for some $\lambda \ge 0$. Furthermore, the greatest lower
bound of all such $\lambda$ coincides with $\lambda^*(f, \varphi)$.
\end{remark}

It is easy to verify that the proposition below, that describes another approach to the definition of exact penalty
function, holds true.

\begin{proposition} \label{PrpEquivDefOfExactPenFunc}
Let the penalty function $F_{\lambda}$ be exact. Then
\begin{equation} \label{ExPenParRepr}
  \lambda^*(f, \varphi) = \sup_{x \in A \setminus \Omega} \frac{f^* - f(x)}{\varphi(x)}.
\end{equation}
Moreover, the penalty function $F_{\lambda}$ is exact if and only if
the supremum on the right-hand side of (\ref{ExPenParRepr}) is finite.
\end{proposition}

\begin{remark}
Unlike the case of local exactness (cf.~Remark~\ref{RmrkLocalExPenPar} above), any globally optimal solution of the
problem $(\mathcal{P})$ is a point of global minimum of the penalty function $F_{\lambda}$ on the set $A$ when 
$\lambda = \lambda^*(f, \varphi)$. However, in the general case, the equality (\ref{EquivOfOptimProblems}) holds true
only if $\lambda > \lambda^*(f, \varphi)$.
\end{remark}

Let us also mention here one well-known property of penalty functions that will be used in the following sections.

\begin{proposition} \label{PrpConvToZero}
Let there exist $\lambda_0 \ge 0$ such that for any $\lambda \ge \lambda_0$ the penalty function $F_{\lambda}$ attains
a global minimum on the set $A$, i.e. $[\lambda_0, +\infty) \subset \dom G$. Then for any selection $x(\cdot)$ of the
set--valued mapping $G(\cdot)$ one has
\begin{equation} \label{ConvToZero}
  \varphi(x(\lambda)) \to 0 \text{ as } \lambda \to \infty.
\end{equation}
\end{proposition}

\begin{proof}
Arguing by reductio ad absurdum, suppose that (\ref{ConvToZero}) does not hold true. Then there exist a selection
$x(\cdot)$ of $G(\cdot)$, $m > 0$, and an increasing sequence $\{ \lambda_n \}$ such that $\lambda_n \to \infty$ and
$\varphi(x(\lambda_n)) \ge m$ as $n \to \infty$. From the first statement of Proposition~\ref{PrpPropertOfPenFunc} it
follows that $f(x(\lambda_n)) \ge f(x(\lambda_1))$ for all $n \in \mathbb{N}$. Therefore
$$
  F_{\lambda_n}(x(\lambda_n)) = f(x(\lambda_n)) + \lambda_n \varphi(x(\lambda_n)) \ge f(x(\lambda_1)) + \lambda_n m,
$$
which implies that
\begin{equation} \label{ConvToInfty}
  F_{\lambda_n}(x(\lambda_n)) \to \infty \text{ as } n \to \infty,
\end{equation}
since $\lambda_n \to \infty$ as $n \to \infty$. On the other hand, for any $y \in \Omega$ such that $f(y) < +\infty$ one
has
$$
  F_{\lambda_n}(x(\lambda_n)) \le F_{\lambda_n}(y) = f(y) < + \infty \quad \forall n \in \mathbb{N},
$$
which contradicts (\ref{ConvToInfty}).  
\end{proof}

\subsection{Necessary and Sufficient Conditions for Exact Penalization: Specific Cases}

We start with the convex case. Taking into account the fact that any local minimum of a convex function is a global one
we obtain that in the convex case the penalty function $F_{\lambda}$ is exact iff it is exact at one of globally
optimal solutions of the initial problem.

\begin{proposition}
Let $X$ be a topological vector space, $M, A \subset X$ be convex sets, and the functions $f$ and $\varphi$ be convex
on the set $A$. Then for the penalty function $F_{\lambda}$ to be exact it is necessary and sufficient that there exists
a globally optimal solution $x^*$ of the problem $(\mathcal{P})$ such that 
$F_{\lambda}$ is exact at $x^*$. Moreover, if $F_{\lambda}$ is exact, then 
$\lambda^*(f, \varphi) = \lambda^*(x^*)$ for any globally optimal solution $x^*$ of the problem $(\mathcal{P})$.
\end{proposition}

\noindent\textit{Remark.}
{In the convex case one can utilize the well-known optimality condition $0 \in \partial F_{\lambda}(x^*)$ in order to
prove the exactness of the penalty function $F_{\lambda}$. Namely, it is easy to check that in the convex case the
penalty function $F_{\lambda}$ is exact iff there exist a globally optimal solution $x^*$ of the problem $(\mathcal{P})$
and $\lambda \ge 0$ such that $0 \in \partial F_{\lambda}(x^*)$. Furthermore, the greatest lower bound of all such
$\lambda \ge 0$ coincides with the least exact penalty parameter $\lambda^*(f, \varphi)$. Note also that the existence
of such $\lambda$ usually follows from the existence of a Lagrange multiplier corresponding to a globally optimal
solution of the problem $(\mathcal{P})$. In particular, one can verify that the validity of Slater's condition for the
problem
\begin{equation} \label{ConvexProblemExample}
  \min f(x) \quad \text{subject to} \quad g_i(x) \le 0, \quad i \in I = \{ 1, \ldots, m \}, \quad x \in A,
\end{equation}
where the functions $f$ and $g_i$, $i \in I$, and the set $A$ are convex, implies that the penalty functions
$$
  F_{\lambda}(x) = f(x) + \lambda \sum_{i = 1}^m \max\{ 0, g_i(x) \}, \:
  H_{\lambda}(x) = f(x) + \lambda \max\{ 0, g_1(x), \ldots, g_m(x) \}
$$
are exact. Furthermore, in this case one can easily estimate the least exact penalty parameter $\lambda^*(f, \varphi)$
via Lagrange multipliers corresponding to optimal solutions of the problem
\eqref{ConvexProblemExample} (cf.~\cite{Bertsekas}). Let us also note that somewhat similar results on the exactness of
penalty functions can be obtained for some DC optimization problems with inequality constraints \cite{ThiDinhNgai}.
}
\vspace{2mm}

In the general case, it is clear that if the penalty function $F_{\lambda}$ is exact, then it is exact at every globally
optimal solution of the problem $(\mathcal{P})$, and
$$
  \lambda^*(x^*) \le \lambda^*(f, \varphi) \quad \forall x^* \in \argmin_{x \in \Omega} f(x).
$$
This condition becomes sufficient in the case when the set $A$ is compact.

\begin{theorem} \label{ThExPenFuncCompactSet}
Let $X$ be a topological space, $A \subset X$ be a compact set, and let the functions $f$ and $\varphi$ be l.s.c. 
on $A$. Then the penalty function $F_{\lambda}$ is exact if and only if it is exact at every globally optimal solution
of the problem $(\mathcal{P})$.
\end{theorem}

\begin{proof}
Suppose that $F_{\lambda}$ is exact at every globally optimal solution of the problem $(\mathcal{P})$. Taking into
account the facts that $A$ is compact, and the functions $f$ and $\varphi$ are l.s.c. one obtains that for any 
$\lambda \ge 0$ the penalty function $F_{\lambda}$ attains a global minimum on $A$, i.e. $\dom G = \mathbb{R}_+$.

Choose a selection $x(\cdot)$ of the set-valued mapping $G(\cdot)$. One can consider $x(\cdot)$ as a net 
$\{ x_{\lambda} \}$, $\lambda \in \Lambda = \mathbb{R}_+$, in $A$, where $x_{\lambda} = x(\lambda)$.
The set $A$ is compact. Therefore there exists a subnet $y_{\gamma}$, $\gamma \in \Gamma$ of the net 
$\{ x_{\lambda} \}$ that converges to some $x^* \in A$ (see, e.g., \cite{Kelley}, Theorem 5.2). 
By Proposition~\ref{PrpConvToZero} the net $\{ \varphi(x_{\lambda}) \}$, $\lambda \in \Lambda$ converges to zero; hence,
its subnet $\{ \varphi(y_{\gamma}) \}$, $\gamma \in \Gamma$ also converges to zero. Thus, taking into
account the lower semicontinuity of $\varphi$ on $A$ one gets that $\varphi(x^*) = 0$ or, equivalently, $x^* \in M$,
since $M = \{ x \in X \mid \varphi(x) = 0 \}$. Consequently, $x^* \in \Omega = M \cap A$.

Let us show that $x^*$ is a globally optimal solution of the problem $(\mathcal{P})$. Indeed, from the facts that
$y_{\gamma}$ is a point of global minimum of $F_{\lambda}$ on $A$ for some $\lambda \ge 0$, and the function $\varphi$
is nonnegative it follows that $f(y_{\gamma}) \le f^*$ for any $\gamma \in \Gamma$. Therefore, applying the lower
semicontinuity of $f$ on $A$ one gets
$$
  f(x^*) \le \lim f(y_{\gamma}) \le f^*,
$$
which yields $f(x^*) = f^*$ by the fact that $x^* \in \Omega$. Thus, $x^*$ is a globally optimal solution of the problem
$(\mathcal{P})$, and $F_{\lambda}$ is exact at $x^*$. 

Fix an arbitrary $\lambda_0 > \lambda^*(x^*)$. Then by Proposition~\ref{PrpUniformLocalPenal} there exists a
neighbourhood $U$ of the point $x^*$ such that
$$
  F_{\lambda}(x^*) \le F_{\lambda}(x) \quad \forall x \in U \cap A \quad \forall \lambda \ge \lambda_0.
$$
From the fact that the net $\{ y_{\gamma} \}$ converges to $x^*$ it follows that there exists $\gamma_0 \in \Gamma$ such
that $y_{\gamma} \in U$ for any $\gamma \ge \gamma_0$, $\gamma \in \Gamma$. Therefore for any $\lambda \ge \lambda_0$
the following inequality holds true:
$$
  F_{\lambda}(x^*) \le F_{\lambda}(y_{\gamma}) \quad \forall \gamma \ge \gamma_0.
$$
By the definition of subnet there exists a function $h \colon \Gamma \to \Lambda$ that is monotone, cofinal (i.e. for
any $\lambda \in \Lambda$ there exists $\gamma \in \Gamma$ for which $h(\gamma) \ge \lambda$) and such that
$y_{\gamma} = x_{h(\gamma)}$ for all $\gamma \in \Gamma$. Hence there exists $\gamma_1 \in \Gamma$ such that
$h(\gamma_1) \ge \lambda_0$, since $h$ is cofinal. Let $\gamma^*$ be an upper bound of the pair 
$\{ \gamma_0, \gamma_1 \}$. Then
$$
  F_{h(\gamma)}(x^*) \le F_{h(\gamma)}(y_{\gamma}) = F_{h(\gamma)}(x_{h(\gamma)}) \quad
  \forall \gamma \ge \gamma^*,
$$
which implies that $x^* \in \Omega$ is a point of global minimum of the function $F_{h(\gamma)}$ for any 
$\gamma \ge \gamma^*$. Consequently, the penalty function $F_{\lambda}$ is exact by Corollary~\ref{CrlrDefExPenFunc}.  
\end{proof}

However, if the set $A$ is not compact, then the exactness of the penalty function $F_{\lambda}$ at every globally
optimal solution of the problem $(\mathcal{P})$ is not sufficient for $F_{\lambda}$ to be an exact penalty function even
in the one-dimensional case.

\begin{example} \label{ExmplNonregFuncs}
Let $X = A = \mathbb{R}$, and $M = \Omega = (- \infty, 0]$. Set $f(x) = - x$, when $x \le 1$. Define $f(x) = - 1 / n$,
when $x \in ( (n - 1)^2 + 1, n^2 ]$, and
$$
  f(x) = -\frac{1}{n-1} + \left(\frac{1}{n - 1} - \frac{1}{n} \right) \big(x - (n - 1)^2\big),
$$
when $x \in ( (n - 1)^2, (n - 1)^2 + 1 ]$ for any $n \in \mathbb{N}$ such that $n \ge 2$. Thus, the function $f$ is
constant on the intervals $( (n - 1)^2 + 1, n^2 ]$ (namely, $f(x) = - 1 / n$ on these intervals), and these ``steps''
are connected by linear functions in such a way that $f$ is a continuous piecewise linear function.

Define
$$
  \varphi(x) = \begin{cases}
    0, \text{ if } x \in (-\infty, 0], \\
    x, \text{ if } x \in (0, 1], \\
    1 / x, \text{ if } x \in [1, +\infty).
  \end{cases}
$$
Note that both functions $f$ and $\varphi$ are Lipschitz continuous on $\mathbb{R}$ with a Lipschitz constant $L = 1$.

The only globally optimal solution of the problem $(\mathcal{P})$ is the point $x^* = 0$. A direct
computation shows that $\overline{\lambda}(0) = 1$. Hence, as it is natural to expect, the penalty function
$F_{\lambda}$ is exact at the point $0$, and the least exact penalty parameter at this point equals $1$.

Let us compute a global minimum of the function $F_{\lambda}$. Fix $\lambda = k \in \mathbb{N}$ with $k \ge 2$. At
first, note that $f(x) = - \varphi(x)$ for all $x \in [0, 1]$, while $f(x) \ge -1$ and $\varphi(x) \ge 1 / 2$ for 
any $x \in [1, 2]$, which implies
\begin{equation} \label{example1}
  F_k(x) \ge F_2(x) \ge 0 \quad \forall x \le 2.
\end{equation}
Observe that $f(x) \ge - 1 / n$ for any $x \in [ (n - 1)^2 + 1, n^2 + 1 ]$ and $n \ge 2$, and
$$
  \frac{1}{n} \le \frac{k}{x} \quad \forall x \in [ (n - 1)^2 + 1, n^2 + 1 ] \quad \forall n \in \{ 2, \ldots, k - 1 \}.
$$
Therefore for any $n \in \{ 2, \ldots, k - 1 \}$ one has
\begin{equation} \label{example2}
  F_k(x) = f(x) + \frac{k}{x} \ge - \frac{1}{n} + \frac{k}{x} \ge 0 \quad 
  \forall x \in [ (n - 1)^2 + 1, n^2 + 1 ].
\end{equation}
On the other hand, if $x \in [ (k - 1)^2 + 1, k^2 ]$, then
\begin{equation} \label{example3}
  F_k(x) = - \frac{1}{k} + \frac{k}{x} \ge - \frac{1}{k} + \frac{k}{k^2} = 0.
\end{equation}
Combining (\ref{example1})--(\ref{example3}) one obtains that
$$
  F_k (x) = f(x) + k \varphi(x) \ge 0 \quad \forall x \in (-\infty, k^2].
$$
It is easy to check that for any $n \in \mathbb{N}$, $n > k \ge 2$ one has
$$
  (F_k)'(x) = f'(x) + k \varphi'(x) = \frac{1}{n(n - 1)} - \frac{k}{x^2} > 0 \quad 
  \forall x \in \big( (n - 1)^2, (n - 1)^2 + 1 \big),
$$
and $(F_k)'(x) = - k / x^2 < 0$ for all $x \in ((n - 1)^2 + 1, n^2)$. Hence the function $F_k$ strictly increases on 
$( (n - 1)^2, (n - 1)^2 + 1 )$, and strictly decreases on the segment $((n - 1)^2 + 1, n^2)$, when $n > k$. Therefore
the penalty function $F_k$ attains a global minimum on the segment $[(n - 1)^2, n^2]$ at one of the endpoints of this
segment for any $n > k$. Consequently, taking into account the fact that
$$
  F_k(n^2) = - \frac{1}{n} + \frac{k}{n^2} = \frac{k - n}{n^2}, \quad
  \min_{n \in \mathbb{N}} \frac{k - n}{n^2} = - \frac{1}{2k} + \frac{k}{(2k)^2} = - \frac{1}{4k} < 0,
$$
one obtains that the penalty function $F_k$ has the unique point of global minimum $x_k = 4 k^2$, if $k \ge 2$. Thus,
$F_{\lambda}$ is not an exact penalty function. 

Note that in this example $x_k \to \infty$ as $k \to \infty$, i.e. a point of global minimum of $F_{\lambda}$ tends to
infinity as $\lambda \to \infty$ or, in other words, there is no bounded selection of the set--valued mapping
$G(\cdot)$.
\end{example}

\begin{remark}
The example above disproves Theorem~3.4.2 from~\cite{Demyanov}.
\end{remark}

\subsection{Non-degeneracy of a Penalty Function}

Example \ref{ExmplNonregFuncs} motivates us to give the definition of a non-degenerate penalty function.

\begin{definition}
Let $X$ be a normed space. The penalty function $F_{\lambda} = f + \lambda \varphi$ is said to be
\textit{non-degenerate}, if there exists $\lambda_0 \ge 0$ such that for any $\lambda \ge \lambda_0$ the penalty
function $F_{\lambda}$ attains a global minimum on the set $A$, and there exists a selection $x(\cdot)$ of the
set-valued mapping $G(\cdot)$ such that the set $\{ x(\lambda) \mid \lambda \ge \lambda_0 \}$ is bounded.
\end{definition}

Roughly speaking, the non-degeneracy condition means that for any sufficiently large $\lambda \ge 0$ the penalty
function $F_{\lambda}$ attains a global minimum on the set $A$, and its global minimizers on $A$ do not escape to
infinity as $\lambda \to + \infty$.

\begin{remark}
The definition of non-degeneracy can be easily transformed to the case when $X$ is a topological vector space or a
metric space. In the latter case, one must assume that the set $\{ x(\lambda) \mid \lambda \ge \lambda_0 \}$ has finite
diameter. One can reformulate main results about the non-degeneracy of a penalty function discussed below to these more
general cases.
\end{remark}

It is obvious that if the set $A$ is bounded, and $F_{\lambda}$ attains a global minimum on $A$ for all sufficiently
large $\lambda \ge 0$, then $F_{\lambda}$ is non-degenerate. Let us describe some other simple sufficient conditions for
non-degeneracy in the case when the set $A$ is unbounded. These conditions are based either on a nonlocal error bound
for the penalty term $\varphi$ or on the boundedness of some sublevel sets.

\begin{proposition}
Let $X$ be a normed space, and let there exist $\lambda_0 \ge 0$ such that for any $\lambda \ge \lambda_0$ the
penalty function $F_{\lambda}$ attains a global minimum on the set $A$. Suppose that the set $\Omega$ is bounded and
there exist $\delta > 0$, and a strictly increasing function $\eta \colon \mathbb{R}_+ \to \mathbb{R}_+$
such that $\eta(0) = 0$ and
$$
  \varphi(x) \ge \eta( d(x, \Omega) ) \quad \forall x \in \Omega_{\delta} = \{ x \in A \mid \varphi(x) < \delta \}.
$$
Then $F_{\lambda}$ is non-degenerate.
\end{proposition}

\begin{proof}
By virtue of Proposition~\ref{PrpConvToZero}, for any selection $x(\cdot)$ of the set-valued mapping $G(\cdot)$, one has
$\varphi(x(\lambda)) \to 0$ as $\lambda \to \infty$. Hence there exists $\lambda_1 \ge \lambda_0$ such that
$$
  \eta( d(x(\lambda), \Omega) ) \le \varphi(x(\lambda)) < \min\{ \delta, \eta(1) \} \quad
  \forall \lambda \ge \lambda_1.
$$
Consequently, $d(x(\lambda), \Omega) < 1$ for all $\lambda > \lambda_1$ by virtue of the fact that $\eta$ is a strictly
increasing function. Therefore taking into account the boundedness of $\Omega$ one obtains that
$$
  \| x(\lambda) \| \le  \sup_{x \in \Omega} \| x \| + 1 < + \infty \quad \forall \lambda \ge \lambda_1.
$$
Thus, $F_{\lambda}$ is non-degenerate.  
\end{proof}

Recall that a function $g \colon X \to \mathbb{R} \cup \{ + \infty \}$ defined on a normed space $X$ is called
\textit{coercive} with respect to the set $A$, if $g(x) \to \infty$ as $\| x \| \to \infty$ such that $x \in A$. It is
easy to check that the following proposition holds true.

\begin{proposition}
Let $X$ be a normed space, and let there exist $\lambda_0 \ge 0$ such that for any $\lambda \ge \lambda_0$ the
penalty function $F_{\lambda}$ attains a global minimum on the set $A$. Suppose that one of the following conditions
is satisfied:
\begin{enumerate}
\item{the set $\{ x \in A \mid F_{\mu}(x) < f^* \}$ is bounded for some $\mu \ge 0$;}

\item{the set $\Omega_{\delta} = \{ x \in A \mid \varphi(x) < \delta \}$ is bounded for some $\delta > 0$ }
\end{enumerate}
(in particular, one can suppose that one of the functions $f$, $\varphi$ or $F_{\mu}$ for some $\mu > 0$ is coercive
with respect to the set $A$). Then $F_{\lambda}$ is non-degenerate.
\end{proposition}

\subsection{Exact Penalization in the Finite Dimensional Case}

In the finite dimensional case, the non-degeneracy of the penalty function along with the exactness of $F_{\lambda}$ at
every globally optimal solution of the initial problem is necessary and sufficient for $F_{\lambda}$ to be exact. The
proof of this results, in essence, repeats the proof of Theorem~\ref{ThExPenFuncCompactSet}.

\begin{theorem} \label{ThExPenFiniteDim}
Let $X$ be a finite dimensional normed space, $A \subset X$ be a closed set, and let the functions $f$ and
$\varphi$ be l.s.c. on $A$. Then the penalty function $F_{\lambda}$ is exact if and only if
the following two conditions are satisfied:
\begin{enumerate}
\item{$F_{\lambda}$ is non-degenerate;}

\item{$F_{\lambda}$ is exact at every globally optimal solution $x^*$ of the problem $(\mathcal{P})$.}
\end{enumerate}
\end{theorem}

\begin{proof}
Necessity. From the definition of global exactness it is obviously follows that the penalty function $F_{\lambda}$ is
exact at every globally optimal solution $x^*$ of the problem $(\mathcal{P})$.

Fix a selection $x(\cdot)$ of $G(\cdot)$, and a globally optimal solution $x^*$ of the problem
$(\mathcal{P})$. From the exactness of the penalty function it follows that for any $\lambda \ge \lambda^*(f, \varphi)$
the point $x^*$ is a global minimizer of $F_{\lambda}$ on $A$. Define
$$
  x_0(\lambda) = \begin{cases}
    x(\lambda), \text{ if } \lambda \in [0, \lambda^*(f, \varphi)) \cap \dom G, \\
    x^*, \text{ if } \lambda \ge \lambda^*(f, \varphi).
  \end{cases}
$$
Then $x_0(\cdot)$ is a selection of $G(\cdot)$, and the set $\{ x_0(\lambda) \mid \lambda \ge \lambda_0 \}$ with
$\lambda_0 = \lambda^*(f, \varphi)$ is bounded. Hence $F_{\lambda}$ is non-degenerate.

Sufficiency. By the definition of non-degeneracy there exist $\lambda_0 \ge 0$ and a selection $x(\cdot)$ of
the set-valued mapping $G(\cdot)$ such that the set $\{ x(\lambda) \mid \lambda \ge \lambda_0 \}$ is bounded. Choose an
increasing sequence $\{ \lambda_n \}$ such that $\lambda_n \to \infty$ as $n \to \infty$ and $\lambda_n \ge \lambda_0$
for all $n$. The corresponding sequence of global minimizers $\{ x(\lambda_n) \} \subset A$ is bounded. Then by virtue
of the fact that the normed space $X$ is finite dimensional there exists a subsequence of the sequence 
$\{ x(\lambda_n) \}$ converging to some $x^*$. Without loss of generality one can suppose that the sequence 
$\{ x(\lambda_n) \}$ itself converges to $x^*$. Moreover, $x^* \in A$ due to the fact that the set $A$ is closed.

By Proposition~\ref{PrpConvToZero} one has $\varphi(x(\lambda_n)) \to 0$ as $n \to \infty$. Therefore
$\varphi(x^*) = 0$, since $\varphi$ is l.s.c.. Hence $x^* \in M$, which implies $x^* \in \Omega = M \cap A$.

Let us check that $x^*$ is a globally optimal solution of the problem $(\mathcal{P})$. Indeed, from the facts that
$x(\lambda_n)$ is a point of global minimum of $F_{\lambda_n}$ on $A$, and the function $\varphi$ is nonnegative it
follows that $f(x(\lambda_n)) \le f^*$ for all $n \in \mathbb{N}$. Therefore taking into account the fact that $f$ is
l.s.c. one gets that
$$
  f(x^*) \le \lim_{n \to \infty} f(x(\lambda_n)) \le f^*.
$$
Hence $x^*$ is a globally optimal solution of the problem $(\mathcal{P})$. Consequently, 
$F_{\lambda}$ is exact at $x^*$.

Fix an arbitrary $\mu > \lambda^*(x^*)$. Then by Proposition~\ref{PrpUniformLocalPenal} there exists $r > 0$ such
that for any $\lambda \ge \mu$ one has
$$
  F_{\lambda}(x^*) \le F_{\lambda}(x) \quad \forall x \in B(x^*, r) \cap A.
$$
Applying the fact that the sequence $\{ x(\lambda_n) \}$ converges to $x^*$ one gets that there exists 
$n_0  \in \mathbb{N}$ such that $x(\lambda_n) \in B(x^*, r)$ for any $n \ge n_0$. Therefore
$$
  F_{\lambda}(x^*) \le F_{\lambda}(x(\lambda_n)) \quad \forall n \ge n_0 \quad \forall \lambda \ge \mu.
$$
Since $\lambda_n \to \infty$ as $n \to \infty$, there exists $n_1 \in \mathbb{N}$ such that for any $n \ge n_1$ one has
$\lambda_n \ge \mu$. Hence
$$
  F_{\lambda_n}(x^*) \le F_{\lambda_n}(x(\lambda_n)) \quad \forall n \ge \max\{ n_0, n_1 \},
$$
which implies that $x^* \in \Omega$ is a point of global minimum of the penalty function $F_{\lambda_n}$ for any
sufficiently large $n$. Consequently, $F_{\lambda}$ is an exact penalty function by Corollary~\ref{CrlrDefExPenFunc}.  
\end{proof}

\begin{remark}
One can verify that the theorem above holds true in the case when $X$ is a Montel space or a metric space such that any
subset $K$ of $X$ having finite diameter is relatively compact. Moreover, one can easily extend the previous theorem
to the general case when $X$ is a topological space. Namely, one must replace the non-degeneracy assumption of 
the theorem with the assumptions that there exists $\lambda_0 \ge 0$ such that $F_{\lambda}$ attains a global minimum 
on $A$ for any $\lambda \ge \lambda_0$, and there exists a selection $x(\cdot)$ of the multifunction $G(\cdot)$ such
that the set $\{ x(\lambda) \mid \lambda \ge \lambda_0 \}$ is relatively compact. However, this result is of theoretical
value only, since, unlike the non-degeneracy condition in the finite dimensional case, the assumption that the set 
$\{ x(\lambda) \mid \lambda \ge \lambda_0 \}$ is relatively compact is extremely hard to verify.
\end{remark}

\subsection{Exact Penalization in Normed Spaces}

Our aim, now, is to show that Theorem~\ref{ThExPenFiniteDim} holds true only in the finite dimensional case, and to
understand additional assumptions that must be made in order to obtain a similar result in the case when $X$ is an
infinite dimensional normed space.

Recall that $\ell_2$ is a linear space of all sequences $x = \{ x_n \}_{n \in \mathbb{N}} \subset \mathbb{R}$ such that
$\sum_{n = 1}^{\infty} |x_n|^2 < +\infty$ endowed with the norm
$$
  \| x \| = \Big( \sum_{n = 1}^{\infty} |x_n|^2 \Big)^{\frac12}.
$$
In the examples below, $X = A = \ell_2$.

\begin{example} \label{ExNotStrongReg}
Let $\Omega = \{ 0 \}$. Set
$$
  \varphi(x) = \sum_{n = 1}^{\infty} \frac{1}{n^2} |x_n| \quad \forall x \in \ell_2.
$$
The function $\varphi$ is non-negative, continuous, positively homogeneous, convex and such that $\varphi(x) = 0$
iff $x = 0$. 

For any $n \in \mathbb{N}$ define a function $s_n \colon \ell_2 \to \mathbb{R}$ such that $s_n(x) = \min\{ x_n, 2 \}$,
if
$x_n \ge 1$, and $s_n(x) = 0$ otherwise. It is easy to see that $s_n$ is upper semicontinuous. Define
$$
  f(x) = \min_{n \in \mathbb{N}} \left( - \frac{1}{n} s_n(x) \right)
$$
(the minimum on the right-hand side is attained, since for any $x \in \ell_2$ there exists $n_0 \in \mathbb{N}$ such
that $s_n(x) = 0$ for all $n \ge n_0$). One can verify that $f$ is l.s.c. on $\ell_2$. Note that $f(x) = 0$ for all $x$
in the open ball $U(0, 1)$. Hence, obviously, $\overline{\lambda}(0) = 0$, and the penalty function $F_{\lambda}$ is
exact at the origin.

Fix an arbitrary $m \in \mathbb{N}$. Observe that for any $\lambda \in [m - 1, m)$ and $x \in \ell_2$ one has
$$
  - \frac1n s_n(x) + \frac{\lambda}{n^2} |x_n| \ge 0 \quad \forall n < m, 
$$
and
$$
  \left(\frac{\lambda}{n^2} - \frac1n \right) < 0, \quad
  - \frac1n s_n(x) + \frac{\lambda}{n^2} |x_n| \ge \left( \frac{\lambda}{n^2} - \frac1n \right)2 
  \quad \forall n \ge m.
$$
Moreover, the latter inequality turns into an equality for any $x \in \ell_2$ such that $x_n = 2$. Therefore
$$
  \inf_{x \in \ell_2} F_{\lambda}(x) = 
  \min_{n \in \mathbb{N}} \min_{x \in \ell_2} \left( - \frac{1}{n} s_n(x) + \lambda \varphi(x) \right) =
  \min_{n \in \mathbb{N}} \left( \frac{2 \lambda}{n^2} - \frac2n \right) < 0 \quad 
  \forall \lambda > 0,
$$
and the infimum on the left-hand side is attained at a point $x_{\lambda}$, where $(x_{\lambda})_{\overline{n}} = 2$,
$(x_{\lambda})_n = 0$, if $n \ne \overline{n}$, and $\overline{n} \in \mathbb{N}$ is such that
$$
  \min_{n \in \mathbb{N}} \left( \frac{2 \lambda}{n^2} - \frac2n \right) = 
  \frac{2 \lambda}{\overline{n}^2} - \frac{2}{\overline{n}}.
$$
Note that $\| x_{\lambda} \| = 2$. Thus, $F_{\lambda}$ is non-degenerate, but it is not an exact penalty function.
\end{example}

In the previous example, one has $d(x_{\lambda}, \Omega) = 2$ for all $x_{\lambda} \in G(\lambda)$ and for any
$\lambda > 0$. This example motivates us to give the following definition of strong non-degeneracy.

\begin{definition}
Let $X$ be a normed space. The penalty function $F_{\lambda}$ is referred to as \textit{strongly non-degenerate}, if
there exists a selection $x(\cdot)$ of the mapping $G(\cdot)$ such that the set
$\{ x(\lambda) \mid \lambda \ge \lambda_0 \}$ is bounded for some $\lambda_0 \ge 0$, and
$\liminf_{\lambda \to \infty} d(x(\lambda), \Omega) = 0$.
\end{definition}

As the following example shows, even the strong non-degeneracy of the penalty function $F_{\lambda}$ along with its
exactness at every globally optimal solution of the initial problem is not sufficient for $F_{\lambda}$ to be
exact.

\begin{example} \label{ExUnbLocalExParam}
Let $\Omega = B(0, 1)$, and
$$
  \varphi(x) = \max\{ \| x \| - 1, 0 \}.
$$
Note that $\varphi$ is a continuous convex function and $\varphi(x) = d(x, \Omega)$.

For any $n \in \mathbb{N}$ define $s_n(x) = \min\{ 0, \max\{ -n(x_n - 1), - 1 / n \} \}$. It is clear that $s_n$
is Lipschitz continuous on $\ell_2$. Note also that $s_n(x) = 0$ iff $x_n \le 1$. Set
$$
  f(x) = \min_{n \in \mathbb{N}} s_n(x) \quad \forall x \in \ell_2.
$$
It is easy to verify that the function $f$ is locally Lipschitz continuous.

For all $x \in \Omega$ one has $f(x) = 0$. Thus, every $x \in \Omega$ is a point of global minimum of $f$ on $\Omega$. 
If $x \in \Omega$ and $x_n < 1$ for all $n \in \mathbb{N}$, then $f(\cdot) = 0$ in a sufficiently small neighbourhood of
$x$. Hence, for any such $x \in \Omega$ one has $\overline{\lambda}(x) = 0$. On the other hand, if $x \in \Omega$ and
for some $n \in \mathbb{N}$ one has $x_n = 1$, then $x_m = 0$ for any $m \ne n$ and $f(\cdot) = s_n(\cdot)$ in a
neighbourhood of $x$. Therefore, as it is easy to check, $\overline{\lambda}(x) = n$ for any such $x \in \Omega$. Thus,
the penalty function $F_{\lambda} = f + \lambda \varphi$ is exact at every globally optimal solution of 
the problem $(\mathcal{P})$.

Let us find a global minimum of $F_{\lambda}$. For any $n \in \mathbb{N}$ and $\lambda \in (0, n)$ one has
$$
  \min_{x \in \ell_2} (s_n(x) + \lambda \varphi(x)) = -\frac1n + \frac{\lambda}{n^2},
$$
and the minimum on the left-hand side is attained at the unique point $x^{(n)}$, where $x_n^{(n)} = 1 + 1 / n^2$ and 
$x^{(n)}_m = 0$ for any $m \ne n$. Note also that $s_n(x) + \lambda \varphi(x) \ge 0$ for any $x \in \ell_2$, if 
$\lambda \ge n$. Consequently, for any $\lambda > 0$ one has
$$
  \min_{x \in \ell_2} F_{\lambda}(x) = \min_{n \in \mathbb{N}} \min_{x \in \ell_2} (s_n(x) + \lambda \varphi(x)) =
  \min_{n \in \mathbb{N}, n \ge \lambda} \left( -\frac1n + \frac{\lambda}{n^2} \right) < 0,
$$
and the minimum on the left-hand side is attained at the point $x_{\lambda} = x^{(\overline{n})}$ for some 
$\overline{n} \in \mathbb{N}$ with $\overline{n} \ge \lambda$. Hence for any $\lambda \ge 0$ there exists 
$x(\lambda) \in G(\lambda)$ such that
$$
  \| x(\lambda) \| \le 1 + \frac{1}{\lambda^2}, \quad d(x(\lambda), \Omega) \le \frac{1}{\lambda^2} \quad
  \forall \lambda > 0.
$$
Thus, $F_{\lambda}$ is strongly non-degenerate, but it is not exact.
\end{example}

In the previous example one has $\sup_{x^* \in \Omega} \overline{\lambda}(x^*) = +\infty$, i.e. the least exact penalty
parameters at globally optimal solutions of the problem $(\mathcal{P})$ are unbounded from above. However, the
boundedness from above of ``local'' exact penalty parameters is also insufficient for the exactness of a penalty
function.

\begin{example} \label{ExNotLipContNeighb}
Let $\Omega$ and $\varphi$ be as in the previous example. For any $n \in \mathbb{N}$ and $x \in \ell_2$ define
$$
  s_n(x) = \begin{cases}
    \max\{ 0, (2x_n - 1) / n \}, \text{ if } x_n \le 1, \\
    \max\{ -n(x_n - 1) + 1 / n, -1/n \}, \text{ if } x_n \ge 1
  \end{cases}
$$
Note that $s_n$ is Lipschitz continuous on $\ell_2$, $s_n(x) = 1 / n$, if $x_n = 1$, and $s_n(x) = 0$ for any 
$x \in \ell_2$ such that $x_n \le 1 / 2$. Set
$$
  f(x) = \min_{n \in \mathbb{N}} s_n(x) \quad \forall x \in \ell_2.
$$
One can verify that $f$ is locally Lipschitz continuous.

It is easy to check that $x^* \in \Omega$ is a point of global minimum of $f$ on $\Omega$ iff $x^*_n \le 1/2$ for any 
$n \in \mathbb{N}$. Moreover, for any such $x^*$ one has $\overline{\lambda}(x^*) = 0$, since $x^*$ is also a point of
local minimum of the function $f$. Thus, the least exact penalty parameters at globally optimal solutions of the problem
$(\mathcal{P})$ are bounded from above

Arguing in the same way as in Example~\ref{ExUnbLocalExParam} one can easily show that
$$
  \min_{x \in \ell_2} F_{\lambda}(x) = 
  \min_{n \in \mathbb{N}, n \ge 2 \lambda} \left( -\frac1n + \frac{2 \lambda}{n^2} \right) < 0,
$$
and the minimum on the left-hand side is attained at the point $x_{\lambda}$, where 
$(x_{\lambda})_{\overline{n}} = 1 + 2 / \overline{n}^2$, $(x_{\lambda})_m = 0$, if $m \ne \overline{n}$ for some 
$\overline{n} \in \mathbb{N}$ with $\overline{n} \ge 2 \lambda$. Thus, $F_{\lambda}$ is strongly non-degenerate, but
not exact.

Note that in this example, the function $f$ is not Lipschitz continuous in a neighbourhood of the set $\Omega$ of the
form $\{ x \in \ell_2 \mid d(x, \Omega) < r \}$ for any $r > 0$.
\end{example}

\begin{remark}
Combining the ideas of Examples~\ref{ExNotStrongReg}--\ref{ExNotLipContNeighb} one can give an example of the penalty
function that is not exact due to the fact that the uniform lower estimate $\varphi(x) \ge a d(x, \Omega)$ in a
neighbourhood of $\Omega$ does not hold true for any $a > 0$.
\end{remark}

The examples discussed above allow us to formulate necessary and sufficient conditions for a penalty
function to be exact in the case when $X$ is an infinite dimensional normed space. 

\begin{theorem} \label{ThExPenNormedSpace}
Let $X$ be a normed space. Then the penalty function $F_{\lambda}$ is exact if and only if the following two
conditions are satisfied:
\begin{enumerate}
\item{$F_{\lambda}$ is strongly non-degenerate;}

\item{for any $R > 0$ there exist $\lambda^* \ge 0$ and $r > 0$ such that
$$
  \frac{f^* - f(y)}{\varphi(y)} \le \lambda^* \quad 
  \forall y \in \big\{ x \in (A \cap B(0, R)) \setminus \Omega \bigm| d( x, \Omega) < r \big\},
$$
}
where, as above, $f^* = \inf_{x \in \Omega} f(x)$.
\end{enumerate}
\end{theorem}

\begin{proof}
Necessity. Suppose that $F_{\lambda}$ is an exact penalty function. Then by Proposition~\ref{PrpEquivDefOfExactPenFunc}
one has
$$
  \frac{f^* - f(y)}{\varphi(y)} \le \lambda^*(f, \varphi) \quad \forall y \in A \setminus \Omega.
$$
Hence condition 2 is satisfied. The fact that the first condition holds true follows directly from the definition
of exact penalty function.

Sufficiency. Let $x(\cdot)$ be a selection of the mapping $G(\cdot)$ such that the set 
$\{ x(\lambda) \mid \lambda \ge \lambda_0 \}$ is bounded for some $\lambda_0 \ge 0$, and 
$\liminf_{\lambda \to \infty} d(x(\lambda), \Omega) = 0$ that exists due to the strong
non-degeneracy of $F_{\lambda}$. Then there exist $R > 0$ and an increasing unbounded sequence $\{ \lambda_n \}$ such
that
$$
  d(x(\lambda_n), \Omega) \xrightarrow[n \to \infty]{} 0, \quad 
  x(\lambda_n) \in A \cap B(0, R) \quad \forall n \in \mathbb{N}.
$$
From condition 2 it follows that there exist $\lambda^* \ge 0$ and $r > 0$ such that for any $\lambda \ge \lambda^*$ one
has
$$
  f(x^*) = F_{\lambda}(x^*) \le F_{\lambda}(y) \quad 
  \forall y \in \big\{ x \in (A \cap B(0, R)) \setminus \Omega \bigm| d( x, \Omega) < r \big\},
$$
where $x^*$ is a globally optimal solution of the problem $(\mathcal{P})$. Consequently, for sufficiently
large $n \in \mathbb{N}$ one has
$$
  F_{\lambda_n}(x^*) \le F_{\lambda_n}(x(\lambda_n)).
$$
Taking into account the fact that $x(\lambda_n)$ is a point of global minimum of $F_{\lambda_n}$ on $A$ one gets that 
$x^* \in \Omega$ is a point of global minimum of $F_{\lambda_n}$ on $A$ for $n$ large enough. It remains to apply
corollary~\ref{CrlrDefExPenFunc}.  
\end{proof}

One can use a similar technique to the one proposed in subsection \ref{SubSecUpperEstim} in order to simplify 
verification of condition 2 in Theorem~\ref{ThExPenNormedSpace}. Namely, this technique utilizes a nonlocal error bound
for the penalty term $\varphi$.

\begin{proposition} \label{PrpExPenNearOmNormedCase}
Let $X$ be a normed space. Suppose that for any $R > 0$ there exist $r > 0$, and 
functions $\omega, \eta \colon \mathbb{R}_+ \to \mathbb{R}_+$ such that
\begin{enumerate}
\item{$\omega(0) = 0$, $\eta(0) = 0$ and $\eta(t) > 0$ for any $t > 0$;}

\item{$f(y) \ge f^* - \omega(d(y, \Omega))$ for any $y \in (A \cap B(0, R)) \setminus \Omega$ such that 
$d(x, \Omega) < r$;}

\item{$\varphi(x) \ge \eta(d(x, \Omega))$ for any $x \in A \cap B(0, R)$ such that $d(x, \Omega) < r$;}

\item{$\limsup_{t \to + 0} \frac{\omega(t)}{\eta(t)} =: \sigma < +\infty$.}
\end{enumerate}
Then for any $R > 0$ and $\varepsilon > 0$ there exists $\lambda^* < \sigma + \varepsilon$ and $r > 0$ such that
$$
  \frac{f^* - f(y)}{\varphi(y)} \le \lambda^* \quad 
  \forall y \in \big\{ x \in (A \cap B(0, R)) \setminus \Omega \bigm| d( x, \Omega) < r \big\}.
$$
\end{proposition}

\begin{proposition}
Let $X$ be a normed space. Suppose that for any bounded set $C \subset X$ there exists a continuous function 
$\omega \colon \mathbb{R}_+ \to \mathbb{R}_+$ such that
\begin{equation} \label{GenLipCond}
  |f(x) - f(y)| < \omega( \| x - y \| ) \quad \forall x, y \in C
\end{equation}
Then for any $R > 0$ and $r > 0$ there exists a function $\omega \colon \mathbb{R}_+ \to \mathbb{R}_+$ such that
$$
  f(y) \ge f^* - \omega(d(y, \Omega)) \quad 
  \forall y \in \{ z \in B(0, R) \setminus \Omega \mid d(z, \Omega) < r \}.
$$
\end{proposition}

\begin{proof}
Fix arbitrary $R > 0$ and $r > 0$, and denote $U = \{ z \in B(0, R) \setminus \Omega \mid d(z, \Omega) < r \}$.
Suppose that the set $U$ is not empty. Define the bounded set $C = \{ z \in X \mid d(z, U) \le r \}$.
By the assumption of the proposition, there exists a continuous function $\omega \colon \mathbb{R}_+ \to \mathbb{R}_+$
such that (\ref{GenLipCond}) holds true.

Choose an arbitrary $y \in U$. Then $d(y, \Omega) < r$. By definition there exists a sequence $\{ x_n \} \subset \Omega$
such that $\| y - x_n \| \to d(y, \Omega)$ as $n \to \infty$ and $\| y - x_n \| \le r$ for any $n \in \mathbb{N}$.
Therefore $\{ x_n \} \subset C$.

From the fact that $\{ x_n \} \subset \Omega$ it follows that $f(x_n) \ge f^*$ for any $n \in \mathbb{N}$. Hence taking
into account condition (\ref{GenLipCond}) one gets that for any $n \in \mathbb{N}$
$$
  f^* - f(y) = f^* - f(x_n) + f(x_n) - f(y) \le f(x_n) - f(y) \le \omega( \| x_n - y \| ).
$$
Passing to the limit as $n \to \infty$, and applying the continuity of the function $\omega$ one obtains the desired
result.  
\end{proof}

Let us mention one simple sufficient condition for $F_{\lambda}$ to be strongly non-degenerate that also relies on a
nonlocal error bound for the function $\varphi$.

\begin{proposition}
Let $X$ be a normed space, and the penalty function $F_{\lambda}$ be non-degenerate. Suppose that for any 
$R > 0$ there exist $\delta > 0$ and a strictly increasing function $\eta \colon \mathbb{R}_+ \to \mathbb{R}_+$ such
that $\eta(0) = 0$ and
\begin{equation} \label{ErrorBdEstimOmDelt}
  \varphi(x) \ge \eta(d(x, \Omega)) \quad \forall x \in \big\{ y \in A \cap B(0, R) \mid \varphi(x) < \delta  \big\}.
\end{equation}
Then $F_{\lambda}$ is strongly non-degenerate.
\end{proposition}

\begin{proof}
By the definition of non-degeneracy there exists a selection $x(\cdot)$ of the set-valued mapping $G(\cdot)$ such that
the set $\{ x(\lambda) \mid \lambda \ge \lambda_0 \}$ is bounded for sufficiently large $\lambda_0 \ge 0$, which means
that one can choose $R > 0$ such that $\| x(\lambda) \| \le R$ for any $\lambda \ge \lambda_0$. Hence by the
assumption of the proposition there exist $\delta > 0$, and a strictly increasing function 
$\eta \colon \mathbb{R}_+ \to \mathbb{R}_+$ depending on the chosen $R > 0$ and satisfying (\ref{ErrorBdEstimOmDelt}).

Applying Proposition~\ref{PrpConvToZero} one gets that $\varphi(x(\lambda)) \to 0$ as $\lambda \to \infty$, and there
exists $\lambda_1 \ge 0$ such that $\varphi(x(\lambda)) < \delta$ for any $\lambda \ge \lambda_1$. Therefore taking
into account (\ref{ErrorBdEstimOmDelt}) one has
$$
  \eta(d(x(\lambda), \Omega)) \le \varphi(x(\lambda)) \quad \forall \lambda \ge \max\{ \lambda_0, \lambda_1 \}.
$$
Hence $\eta(d(x(\lambda), \Omega)) \to 0$ as $\lambda \to \infty$. Taking into account the facts that $\eta$ is strictly
increasing and $\eta(0) = 0$ it is easy to check that $d(x(\lambda), \Omega) \to 0$ as $\lambda \to \infty$. Thus,
$F_{\lambda}$ is strongly non-degenerate.  
\end{proof}

\subsection{Necessary and Sufficient Conditions for Exact Penalization: General Case}

Theorem~\ref{ThExPenNormedSpace} gives necessary and sufficient conditions for the penalty function $F_{\lambda}$ to be
exact in the case when $X$ is a normed space. However, the assumptions of this theorem are hard to verify. In this
subsection, we discuss some other general necessary and sufficient conditions for the penalty function $F_{\lambda}$ to
be exact that do not rely on the (strong) non-degeneracy of $F_{\lambda}$.

Suppose that there exists $\lambda_0 \ge 0$ such that for any $\lambda \ge \lambda_0$ the penalty function
$F_{\lambda}$ attains a global minimum on the set $A$. Then by Proposition~\ref{PrpConvToZero} for any selection
$x(\cdot)$ of the mapping $G(\cdot)$ one has $\varphi(x(\lambda)) \to 0$ as $\lambda \to \infty$. Hence for any 
$\delta > 0$ one gets that $\varphi(x(\lambda)) < \delta$ for any sufficiently large $\lambda$. Therefore, for the
penalty function $F_{\lambda}$ to be exact it is necessary and sufficient that $F_{\lambda}$ is an exact penalty
function for the problem
$$
  \min f(x) \quad \text{subject to} \quad x \in M, \quad x \in \Omega_{\delta}
$$
for some $\delta > 0$, i.e. the set $A$ can be replaced by the ``smaller'' set $\Omega_{\delta}$. Here 
$\Omega_{\delta} = \{ x \in A \mid \varphi(x) < \delta \}$. However, this result can be significantly sharpened, since
one can replace the assumption that $F_{\lambda}$ attains a global minimum with the assumption that $F_{\lambda}$ is
merely bounded below. In order to give a convenient formulation of this result we need an auxiliary definition of
exactness of a penalty function on a set.

\begin{definition}
Let $C \subset A$ be a nonempty set. The penalty function $F_{\lambda}$ is said to be \textit{exact} on the set $C$, if
there exists $\lambda^* \ge 0$ such that for any $\lambda \ge \lambda^*$ one has $F_{\lambda}(x) \ge f^*$ for any 
$x \in C$. The greatest lower bound of all such $\lambda^*$ is called \textit{the least exact penalty parameter} of
$F_{\lambda}$ on $C$, and is denoted by $\lambda^*(C, f, \varphi)$ or simply by $\lambda^*(C)$, if $f$ and $\varphi$ are
fixed.
\end{definition}

\begin{remark}
It is easy to check that if $F_{\lambda}$ is exact on a set $C \subset A$, then
$$
  \lambda^*(C) = \sup_{y \in C \setminus \Omega} \frac{f^* - f(y)}{\varphi(y)} < + \infty.
$$
Furthermore, the penalty function $F_{\lambda}$ is exact on the set $C$ iff the supremum on the right-hand side is
finite. Note also that $\lambda^*(A) = \lambda^*(f, \varphi)$.
\end{remark}

\begin{lemma}[on exactness of a penalty function] \label{LemmaEPF}
Let $X$ be a topological space. The penalty function $F_{\lambda}$ is exact if and only if the following two conditions
are satisfied:
\begin{enumerate}
\item{$F_{\mu}$ is bounded below on $A$ for some $\mu \ge 0$;}

\item{the penalty function $F_{\lambda}$ is exact on $\Omega_{\delta}$ for some $\delta > 0$.}
\end{enumerate}
Moreover, one has
\begin{equation} \label{ExPenParamViaOmDelta}
  \lambda^*(f, \varphi) \le \max\left\{ \lambda^*(\Omega_{\delta}), \mu + \frac{f^* - c}{\delta} \right\},
\end{equation}
where $c = \inf_{x \in A} F_{\mu}(x)$.
\end{lemma}

\begin{proof}
Let us prove the ``only if'' part of the theorem. The validity of the ``if'' part follows directly from definitions.

Note that $c = \inf_{x \in A} F_{\mu}(x) > - \infty$ due to the fact that $F_{\mu}$ is bounded below on $A$.
Consequently, for any $x \in A \setminus \Omega_{\delta}$ one has
$$
  F_{\lambda}(x) = F_{\mu}(x) + (\lambda - \mu) \varphi(x) \ge c + (\lambda - \mu) \delta \ge
  f^* \quad \forall \lambda \ge \nu,
$$
where $\nu = \mu + (f^* - c) / \delta$. Observe that from the fact that the penalty function $F_{\lambda}$ is
exact on $\Omega_{\delta}$ it follows that
$$
  F_{\lambda}(x) \ge f^* \quad \forall \lambda \ge \lambda^*(\Omega_{\delta}) \quad 
  \forall x \in \Omega_{\delta}.
$$
Therefore for any $x \in A \setminus \Omega$ one has
$$
  F_{\lambda}(x) \ge f^* \quad \forall \lambda \ge 
  \max\left\{ \lambda^*(\Omega_{\delta}), \mu + \frac{f^* - c}{\delta} \right\},
$$
which implies that $F_{\lambda}$ is an exact penalty function, and the inequality (\ref{ExPenParamViaOmDelta}) holds
true.  
\end{proof}

One can easily obtain sufficient conditions for $F_{\lambda}$ to be exact on $\Omega_{\delta}$ similar to the ones
stated in Proposition~\ref{PrpExPenNearOmNormedCase}.

\begin{proposition}
Let $(X, d)$ be a metric space. Suppose that there exist $ \delta > 0$ and functions 
$\omega, \eta \colon \mathbb{R}_+ \to \mathbb{R}_+$ such that
\begin{enumerate}
\item{$\omega(0) = 0$, $\eta(0) = 0$ and $\eta(t) > 0$ for any $t > 0$;}

\item{$f(x) \ge f^* - \omega(d(x, \Omega))$ for any $x \in \Omega_{\delta} \setminus \Omega$;}

\item{$\varphi(x) \ge \eta(d(x, \Omega))$ for any $x \in \Omega_{\delta}$;}

\item{$\sup\{ \omega(t) / \eta(t) \mid t \in \mathbb{R}_+, \eta(t) < \delta \} =: \sigma < +\infty$.}
\end{enumerate}
Then $F_{\lambda}$ is exact on $\Omega_{\delta}$ and $\lambda^*(\Omega_{\delta}) \le \sigma$.
\end{proposition}

\begin{remark}
{(i) Note that the above proposition along with the lemma on the exactness of a penalty function significantly sharpens
some results on exact penalization from \cite{Demyanov} (in particular, Theorem 3.4.1), since it allows one to avoid
any assumptions on the existence of a point of global minimum of the penalty function.
}

\noindent{(ii) If $X$ is a normed space and the set $\{ x \in A \mid F_{\mu}(x) < f^* \}$ is bounded 
for some $\mu \ge 0$, then the assumption that $F_{\lambda}$ is exact on $\Omega_{\delta}$ for some $\delta > 0$ in the
lemma on exactness of a penalty function can be replaced by the weaker assumption that for any $R > 0$ the penalty
function $F_{\lambda}$ is exact on the set $\Omega_{\delta} \cap B(0, R)$. Moreover, one can also modify the previous
proposition accordingly in order to obtain simple sufficient condition for $F_{\lambda}$ to be exact 
on $\Omega_{\delta} \cap B(0, R)$ for any $R > 0$.
}
\end{remark}

As an important corollary to the lemma on exactness of a penalty function, one can obtain necessary and sufficient
conditions for exact penalization in the finite dimensional case different from the ones stated in
Theorem~\ref{ThExPenFiniteDim}. The main difference of these conditions is that one does not suppose that the penalty
function $F_{\lambda}$ attains a global minimum for sufficiently large $\lambda \ge 0$, which makes it easier to verify
these conditions.

\begin{theorem} \label{ThFinDimGenCase}
Let $X$ be a finite dimensional normed space, $A$ be a closed set, $\varphi$ be l.s.c. on $A$, and $f$ be l.s.c. on
$\Omega$. Suppose also that one of the following conditions is satisfied: 
\begin{enumerate}
\item{the set $\{ x \in A \mid F_{\mu}(x) < f^* \}$ is bounded for some $\mu \ge 0$;}

\item{the set $\Omega_{\delta}$ is bounded for some $\delta > 0$.}
\end{enumerate}
Then for the penalty function $F_{\lambda}$ to be exact it is necessary and sufficient that $F_{\lambda}$ is exact
at every globally optimal solution of the problem $(\mathcal{P})$, and $F_{\lambda_0}$ is bounded below on $A$ for
some $\lambda_0 \ge 0$.
\end{theorem}

\begin{proof}
Necessity. The validity of the assertion follows directly from definitions.

Sufficiency. Let us show that $F_{\lambda}$ is exact on $\Omega_{\delta}$ for sufficiently small $\delta > 0$. Then by
the lemma on exactness of a penalty function one concludes that $F_{\lambda}$ is exact. 

It is easy to check that if the set $\Omega_{\delta}$ is bounded for some $\delta > 0$, and $F_{\lambda_0}$ is bounded
below on $A$, then the set $\{ x \in A \mid F_{\mu}(x) < f^* \}$ is bounded for sufficiently large $\mu \ge 0$.
Thus, one can suppose that condition 1 holds true. Therefore there exists $R > 0$ such that for any $\lambda \ge \mu$
one has 
\begin{equation} \label{EqExPenOutsideBall}
  F_{\lambda}(x) \ge f^* \quad \forall x \in X \colon \| x \| > R,
\end{equation}
i.e. the penalty function $F_{\lambda}$ is exact outside the ball $B(0, R)$.

Denote by $\Omega^*$ the set of all points of global minimum of $f$ on $\Omega$ that belong to $B(0, R)$. Observe that
$\Omega$ is closed by virtue of the facts that $\varphi$ is non-negative and l.s.c. on $A$, 
$\Omega = \{ x \in A \mid \varphi(x) = 0 \}$, and $A$ is closed. Hence applying the lower semicontinuity of $f$ on
$\Omega$, and the fact that $\Omega^* \subset \Omega \cap B(0, R)$ one gets that $\Omega^*$ is a compact set (recall
that $X$ is a finite dimensional normed space).

Let us show that there exist $\nu \ge 0$ and $r_0 > 0$ such that
\begin{equation} \label{ExactInANeighbOfOptSol}
  F_{\lambda}(x) \ge f^* \quad \forall x \in \bigcup_{x^* \in \Omega^*} \big( U(x^*, r_0) \cap A \big)
  \quad \forall \lambda \ge \nu,
\end{equation}
i.e. that $F_{\lambda}$ is exact in a neighbourhood of all globally optimal solutions of the problem $(\mathcal{P})$
belonging to the ball $B(0, R)$.

Indeed, by the assumptions of the theorem the penalty function $F_{\lambda}$ is exact at every globally optimal solution
of the problem $(\mathcal{P})$. Therefore for any $x^* \in \Omega^*$ there exist $\lambda(x^*) \ge 0$ and $r(x^*) > 0$ 
such that
\begin{equation} \label{ExactAtOptSolFinDimCase}
  F_{\lambda}(x) \ge F_{\lambda}(x^*) = f^* \quad \forall x \in U\big(x^*, r(x^*)\big) \cap A
  \quad \forall \lambda \ge \lambda(x^*).
\end{equation}
Taking into account the compactness of $\Omega^*$ one gets that there exist $x_1^*, \ldots, x_m^* \in \Omega^*$
such that
$$
  \Omega^* \subset \bigcup_{k = 1}^m U \big( x_k^*, r(x_k^*) / 2 \big).
$$
Denote
$$
  \nu = \max_{k \in \{1, \ldots, m \}} \lambda(x_k^*), \quad
  r_0 = \min_{k \in \{1, \ldots, m \}} r(x_k^*).
$$
Let $x \in U(x^*, r_0) \cap A$ for some $x^* \in \Omega^*$. Then there exists $k \in \{1, \ldots, m\}$ such that 
$x^* \in U(x_k^*, r(x_k^*) / 2) \cap A$. Therefore $x \in U(x_k^*, r(x_k^*)) \cap A$, and applying
(\ref{ExactAtOptSolFinDimCase}) one obtains that
$$
  F_{\lambda}(x) \ge F_{\lambda}(x_k^*) \ge f^* \quad \forall \lambda \ge \nu,
$$
i.e. (\ref{ExactInANeighbOfOptSol}) holds true.

Denote
\begin{equation} \label{DefOfNeighbOfGlobMin}
  U = \bigcup_{x^* \in \Omega^*} \Big( U(x^*, r_0) \cap A \Big), \quad C = (\Omega \setminus U) \cap B(0, R).
\end{equation}
Let $x \in C$. Then $f(x) > f^*$, and taking into account the lower semicontinuity of $f$ on $\Omega$ one gets that
there exists $h(x) > 0$ such that $f(y) > f^*$ for all $y \in U(x, h(x))$. Denote
$$
  V = \bigcup_{x \in C} \Big( U(x, h(x)) \cap A \Big).
$$
Note that for any $x \in V$ one has $f(x) > f^*$, which yields
\begin{equation} \label{ExactOutsideOptSol}
  F_{\lambda}(x) \ge f(x) > f^* \quad \forall x \in V \quad \forall \lambda \ge 0,
\end{equation}
i.e. $F_{\lambda}$ is exact on $V$ and $\lambda^*(V) = 0$.

The sets $U$ and $V$ are obviously open in $A$ (here and below we suppose that the set $A$ is endowed with the induced
metric). Therefore the set 
$$
  K = (A \cap B(0, R)) \setminus (U \cup V)
$$
is closed in $A$, which implies that $K$ is closed in $X$, since $A$ is closed. Consequently, $K$ is compact.
Furthermore, by the definition of $U$ and $V$ one has $\Omega \cap B(0, R) \subset U \cup V$, i.e. the sets $K$ and
$\Omega$ are disjoint. Hence for any $x \in K$ one has $\varphi(x) > 0$, which, with the use of the compactness of $K$,
implies that there exists $\delta > 0$ such that $\varphi(x) \ge \delta$ for any $x \in K$ or, equivalently,
$\Omega_{\delta} \cap B(0, R) \subset U \cup V$.
Therefore, combining (\ref{EqExPenOutsideBall}), (\ref{ExactInANeighbOfOptSol}), (\ref{DefOfNeighbOfGlobMin}), and
(\ref{ExactOutsideOptSol}) one gets that
$$
  F_{\lambda}(x) \ge f^* \quad \forall x \in \Omega_{\delta} \quad \forall \lambda \ge \max\{ \mu, \nu \},
$$
i.e. $F_{\lambda}$ is exact on $\Omega_{\delta}$.
\end{proof}

\begin{remark}
A vague prototype of Theorem~\ref{ThFinDimGenCase} appeared in \cite{WuBaiYangZhang}, Corollary~2.3.
\end{remark}

Arguing in a similar way to the proof of the theorem above one can verify that the following result on the exactness of
the penalty function $F_{\lambda}$ on bounded sets holds true.

\begin{theorem} \label{ThExactOnBoundedSets}
Let $X$ be a finite dimensional normed space, the functions $f$ and $\varphi$ be l.s.c. on $A$, and $A$ be closed. Then
the penalty function $F_{\lambda}$ is exact on any bounded set $C \subset A$ if and only if $F_{\lambda}$ is exact at
every globally optimal solution of the problem $(\mathcal{P})$.
\end{theorem}

\subsection{The Palais--Smale Condition}

An interesting approach to the study of exact penalization based on the use of the Palais--Smale condition
\cite{PalaisSmaleCond} was proposed in the works of Zaslavski (see \cite{Zaslavski} and the references therein). In this
subsection, we generalize this approach to our setting in order to obtain another general method for the study of exact
penalization. Furthermore, this generalization allows us to significantly sharpen some results from \cite{Zaslavski}, in
particular Theorems 2.3--2.5.

In this subsection, we assume that $(X, d)$ is a metric space. In order to introduce a generalized Palais--Smale
condition we need to recall the definition of the rate of steepest descent of a function defined on a metric space
(see, for instance, \cite{Demyanov, DemyanovRSD}). The following definition is a simple modification of the one given in
\cite{DemyanovRSD} that is more suitable for our purposes.

\begin{definition}
Let $g \colon X \to \mathbb{R} \cup \{ +\infty \} \cup \{ - \infty \}$, and $x \in A$ be such that $|g(x)| < +\infty$.
The quantity
$$
  g^{\downarrow}_A(x) = \liminf_{y \to x, y \in A} \frac{g(y) - g(x)}{d(y, x)}
$$
is called \textit{the rate of steepest descent} of the function $g$ on the set $A$ at the point $x$. If $A = X$, then
the value $g^{\downarrow}(x) = g^{\downarrow}_X(x)$ is referred to as \textit{the rate of steepest descent} of $g$ at
$x$.
\end{definition}

\begin{remark}
Note an obvious connection between rate of steepest descent and the much more wide-spread in the literature on
nonsmooth analysis and variational analysis tool, namely, strong slope (see, for instance, \cite{Ioffe,Aze}). Recall
that the quantity
$$
  |\nabla| g(x) = \limsup_{y \to x} \frac{(g(x) - g(y))^+}{d(x, y)}
$$
is called \textit{the strong slope} of $g$ at $x$. Here $t^+ = \max\{ t, 0 \}$. 

It is easy to see that $|\nabla| g(x) > 0$ iff $g^{\downarrow}(x) < 0$, and in the case $|\nabla| g(x) > 0$ the
equality $|\nabla| g(x) = - g^{\downarrow}(x)$ holds true. Therefore almost any result using strong slope can be
easily reformulated in terms of rate of steepest descent; however, the converse statement is not true. Indeed, by
definition one has $|\nabla| g(\cdot) \ge 0$ for any function $g$. On the other hand, the rate of steepest descent
$g^{\downarrow}(x)$ can be greater than zero. Thus, the rate of steepest descent carries more information about 
function's behaviour in some cases.
\end{remark}

We need the following approximate Fermat's rule in terms of rate of steepest descent that is a simple corollary to the
Ekeland variational principle (cf.~\cite{UderzoCalm, Aze}).

\begin{lemma}[Approximate Fermat's rule]
Let $(X, d)$ be a complete metric space, and let a function $g \colon X \to \mathbb{R} \cup \{ + \infty \}$ be proper,
l.s.c., and bounded below. Let also $\varepsilon > 0$ and $x_{\varepsilon}$ be such that
$$
  g(x_{\varepsilon}) \le \inf_{x \in X} g(x) + \varepsilon.
$$
Then for any $r > 0$ there exists $y \in X$ such that $g(y) \le g(x_{\varepsilon})$, 
$d(y, x_{\varepsilon}) \le r$ and $g^{\downarrow}(y) \ge - \varepsilon / r$.
\end{lemma}

Now we are ready to give the definition of generalized Palais--Smale condition. Since we shall apply this condition
to the function $\varphi$, we formulate this condition in a very specific case. However, it is worth mentioning
that our definition is weaker than the traditional one even in the case when $X$ is a normed space, and the function
$\varphi$ is continuously Fr\'echet differentiable.

\begin{definition}
The function $\varphi$ is said to satisfy the \textit{generalized Palais--Smale condition} on the set $A$, if any
sequence $\{ x_n \} \subset A \setminus \Omega$ such that 
$$
  \liminf_{n \to \infty} \varphi(x_n) = 0, \quad \liminf_{n \to \infty} \varphi^{\downarrow}_A(x_n) \ge 0
$$
has a convergent subsequence.
\end{definition}

The following result on exact penalization in the case when $\varphi$ satisfies the generalized Palais--Smale condition
holds true.

\begin{theorem} \label{ThPalaisSmaleCond}
Let $(X, d)$ be a complete metric space, $A$ be a closed set, and the functions $f$ and $\varphi$ be l.s.c. on $A$.
Suppose also that 
\begin{enumerate}
\item{$f$ is Lipschitz continuous on $\Omega_{\delta} \setminus \Omega$ for some $\delta > 0$;}

\item{ $\varphi$ satisfies the generalized Palais--Smale condition.}
\end{enumerate}
Then the penalty function $F_{\lambda}$ is exact if and only if $F_{\lambda}$ is exact at every globally optimal
solution of the problem $(\mathcal{P})$, and $F_{\mu}$ is bounded below on $A$ for some $\mu > 0$.
\end{theorem}

\begin{proof}
Let $F_{\mu}$ be bounded below on $A$ for some $\mu > 0$, and $F_{\lambda}$ be exact at every globally optimal
solution of the problem $(\mathcal{P})$. Arguing by reductio ad absurdum, suppose that the penalty function
$F_{\lambda}$ is not exact. Then for any $\lambda > 0$ one has $\inf_{x \in A} F_{\lambda}(x) < f^*$.

Choose a strictly decreasing sequence $\{ \delta_n \} \subset (0, \delta)$ such that $\delta_n \to 0$ as 
$n \to \infty$. For any $n \in \mathbb{N}$ and $x \in A \setminus \Omega_{\delta_n}$ one has
\begin{equation} \label{BooundFromBelowOutsOmega}
  F_{\lambda}(x) = F_{\mu}(x) + (\lambda - \mu) \varphi(x) \ge c + (\lambda - \mu) \delta_n \ge f^* 
  \quad \forall \lambda \ge \mu + \frac{f^* - c}{\delta_n},
\end{equation}
where $c = \inf_{x \in A} F_{\mu}(x)$. Hence for any $n \in \mathbb{N}$ there exists $x_n \in \Omega_{\delta_n}$ such
that $F_{\lambda_n}(x_n) < f^*$, where $\lambda_n = \mu + (f^* - c) / \delta_n$. Applying approximate Fermat's rule
with $r = \varepsilon = F_{\lambda_n}(x_n) - \inf_{x \in A} F_{\lambda_n}(x)$ one gets that for all $n \in \mathbb{N}$
there exists $y_n \in A$ such that 
\begin{equation} \label{approxFermatRuleCons}
  F_{\lambda_n}(y_n) \le F_{\lambda_n}(x_n) < f^*, \quad (F_{\lambda_n})_A^{\downarrow}(y_n) \ge - 1 
  \quad \forall n \in \mathbb{N}.
\end{equation}
Observe that by (\ref{BooundFromBelowOutsOmega}) one has $y_n \in \Omega_{\delta_n} \setminus \Omega$ or, equivalently, 
$0 < \varphi(y_n) < \delta_n$, which implies that $\varphi(y_n) \to 0$ as $n \to \infty$.

Consider the sequence $\{ \varphi^{\downarrow}_A(y_n) \}$. At first, suppose that 
$\liminf_{n \to \infty} \varphi^{\downarrow}_A(y_n) \le - a < 0$ for some $a > 0$. Then there exists a subsequence which
we denote again by $\{ y_n \}$ such that $\varphi^{\downarrow}_A(y_n) \le - 2a / 3$ for all $n$. By the definition
of rate of steepest descent, for any $n \in \mathbb{N}$ there exists a sequence $\{ z_s^{(n)} \} \subset A$, 
$s \in \mathbb{N}$ such that $z_s^{(n)} \to y_n$ as $s \to \infty$ and
$$
  \varphi(z^{(n)}_s) - \varphi(y_n) \le - \frac{a}{2} d(z^{(n)}_s, y_n) \quad \forall s \in \mathbb{N}. 
$$
Moreover, without loss of generality one can suppose that $\varphi(z^{(n)}_s) > 0$ for any $s \in \mathbb{N}$ due to
the facts that $\varphi$ is l.s.c. on $A$, and $\varphi(y_n) > 0$. 
Hence $z^{(n)}_s \in \Omega_{\delta_n} \setminus \Omega$ for all $s \in \mathbb{N}$. Consequently, employing the
Lipschitz continuity of $f$ on $\Omega_{\delta} \setminus \Omega$ one gets that there exists $L > 0$ such that 
for any $s \in \mathbb{N}$
\begin{multline*}
  F_{\lambda_n}(z^{(n)}_s) - F_{\lambda_n}(y_n) = f(z^{(n)}_s) - f(y_n) + \lambda_n(\varphi(z^{(n)}_s) - \varphi(y_n))
  \\
  \le \left( L - \frac{a}{2}\lambda_n \right) d(z^{(n)}_s, y_n),
\end{multline*}
which yields $(F_{\lambda_n})_A^{\downarrow}(y_n) \le L - a \lambda_n / 2$. Taking into account the fact that the
sequence $\{ \delta_n \}$ converges to zero one obtains that for sufficiently large $n \in \mathbb{N}$ one has
$$
  \lambda_n = \mu + \frac{f^* - c}{\delta_n} > \frac{2 L + 2}{a}.
$$
Therefore, for all $n$ large enough one has $(F_{\lambda_n})^{\downarrow}_A(y_n) < - 1$, which contradicts
(\ref{approxFermatRuleCons}). Thus, $\liminf_{n \to \infty} \varphi^{\downarrow}_A(y_n) \ge 0$ and $\varphi(y_n) \to 0$
as $n \to \infty$. Hence applying the generalized Palais--Smale condition one gets that there exists a subsequence 
$\{ y_{n_k} \}$ converging to some $y_0 \in A$ (recall that $A$ is closed). Taking into account the lower
semicontinuity of $\varphi$ one gets that $\varphi(y_0) = 0$, which implies $y_0 \in \Omega$. 

Note that for any $k \in \mathbb{N}$ one has $f(y_{n_k}) < f^*$, since $F_{\lambda}(x) = f(x)$ for any 
$\lambda \ge 0$ and $x \in \Omega$. Consequently, applying the lower semicontinuity of $f$ and the fact that 
$f^* = \min_{x \in \Omega} f(x)$ one gets that $y_0$ is a globally optimal solution of the problem
$(\mathcal{P})$, and $F_{\lambda}$ is exact at $y_0$. Hence there exists $\lambda_0 > 0$ and $r > 0$ such that 
$$
  F_{\lambda}(x) \ge F_{\lambda}(y_0) = f^* \quad \forall x \in A \cap B(y_0, r) \quad \forall \lambda \ge \lambda_0.
$$
Note that for any sufficiently large $k$ one has $y_{n_k} \in A \cap B(y_0, r)$ and $\lambda_{n_k} \ge \lambda_0$.
Therefore one has
$$
  F_{\lambda_{n_k}}(y_{n_k}) \ge f^*,
$$
which contradicts the definition of $y_{n_k}$. Thus, the penalty function $F_{\lambda}$ is exact.

The validity of the converse statement follows directly from definitions.  
\end{proof}

In the case when $X$ is a normed space one can prove a slightly stronger version of the theorem above. Namely, let us
say that a function $\varphi \colon X \to \mathbb{R}_+$ satisfy the \textit{generalized Palais--Smale condition on
bounded subsets} of the set $A$, if any bounded sequence $\{ x_n \} \subset A \setminus \Omega$ such that 
$$
  \liminf_{n \to \infty} \varphi(x_n) = 0, \quad \liminf_{n \to \infty} \varphi^{\downarrow}_A(x_n) \ge 0
$$
has a convergent subsequence. Arguing in a similar way to the proof of Theorem~\ref{ThPalaisSmaleCond} one can easily
verify that the following result holds true.

\begin{theorem}
Let $X$ be a Banach space, $A$ be a closed set, and the functions $f$ and $\varphi$ be l.s.c. on $A$. Suppose that $f$
is Lipschitz continuous on any bounded subset of $\Omega_{\delta} \setminus \Omega$ for some $\delta > 0$, and
$\varphi$ satisfy the generalized Palais--Smale condition on bounded subsets of the set $A$. Suppose also that one of
the following conditions is satisfied:
\begin{enumerate}
\item{the set $\{ x \in A \mid F_{\mu}(x) < f^* \}$ is bounded for some $\mu > 0$;}

\item{the set $\Omega_{\xi}$ is bounded for some $\xi > 0$.}
\end{enumerate}
Then the penalty function $F_{\lambda}$ is exact if and only if $F_{\lambda}$ is exact at every globally optimal
solution of the problem $(\mathcal{P})$, and $F_{\nu}$ is bounded below on $A$ for some $\nu > 0$.
\end{theorem}

\subsection{The Optimal Value Function of a Perturbed Problem and Exact Penalization}

As in the case of local exactness, the global exactness of a penalty function can be studied via an analysis of a
perturbed optimization problem. This analysis is based on the study of behaviour of the optimal value function of a
perturbed problem. Thus, the main results of this subsection underline the importance of parametric optimization for
the study of constrained optimization problems and, in particular, exact penalty functions. 

Note that the results of this subsection sharpen some results from \cite{LowerOrderCalmness} and some results on linear
penalty functions from \cite{AugmLagrange}.

Following the ideas of subsection~\ref{SectProbCalmness}, consider the perturbed family of constrained optimization
problems
$$
  \min f(x) \quad \text{subject to} \quad x \in M(p), \quad x \in A, \eqno (\mathcal{P}_p)
$$
where $M \colon P \rightrightarrows X$ is a set--valued mapping, $(P, d)$ is a metric space of perturbation
parameters, and $M(p^*) = M$ for some $p^* \in P$. Recall that $\Omega(p) = M(p) \cap A$ for any $p \in P$, and
$\Omega^{-1}(x) = \{ p \in P \mid x \in \Omega(p) \}$ for any $x \in X$. 

Denote by $h(p) = \inf_{x \in \Omega(p)} f(x)$ for any $p \in P$ \textit{the optimal value function} (or
\textit{the perturbation function}) of the perturbed problem ($\mathcal{P}_p$).

\begin{definition}
Let $\omega \colon \mathbb{R}_+ \to \mathbb{R}_+$ be a given function. The optimal value function $h$ is called
$\omega$-\textit{calm from below} at the point $p^*$, if there exist $r > 0$ and $L > 0$ such that
$$
  h(p) - h(p^*) \ge - L \omega(d(p, p^*)) \quad \forall p \in B(p^*, r).
$$
\end{definition}

Necessary and sufficient conditions for the penalty function $F_{\lambda}$ to be exact can be expressed in terms of the
calmness from below of the optimal value function $h(p)$.

\begin{theorem}
Let $X$ be a topological space, and $\omega \colon \mathbb{R}_+ \to \mathbb{R}_+$ be a continuous from the right
function. Suppose that the following conditions are satisfied:
\begin{enumerate}
\item{the penalty function $F_{\mu}$ is bounded below on $A$ for some $\mu > 0$;}

\item{the optimal value function $h$ is $\omega$-calm from below at $p^*$, i.e. there exist $r_1 > 0$ and $L > 0$ 
such that
$$
  h(p) - h(p^*) \ge - L \omega(d(p, p^*)) \quad \forall p \in B(p^*, r_1);
$$
}

\item{there exist $\delta > 0$, $r_2 \in (0, r_1]$ and $a > 0$ such that
\begin{equation} \label{OmegaDeltaPertProb}
  \Omega_{\delta} \subset \bigcup_{p \in U(p^*, r_2)} \Omega(p),
\end{equation}
and
$$
  \varphi(x) \ge a \omega(d(p^*, \Omega^{-1}(x))) \quad \forall x \in \Omega(p) \quad \forall p \in U(p^*, r_2).
$$
}
\end{enumerate}
Then the penalty function $F_{\lambda}$ is exact, and
$$
  \lambda^*(f, \varphi) \le \max\left\{ \frac{L}{a}, \mu + \frac{f^* - c}{\delta} \right\},
$$
where $c = \inf_{x \in A} F_{\mu}(x)$.
\end{theorem}

\begin{proof}
Let us show that $F_{\lambda}$ is exact on $\Omega_{\delta}$ and $\lambda^*(\Omega_{\delta}) \le L / a$, where 
$\delta > 0$ is such that (\ref{OmegaDeltaPertProb}) holds true. Then applying the lemma on exactness of a penalty
function one gets the desired result.

Clearly, $h(p^*) = f^*$, since $M(p^*) = M$. Applying the $\omega$-calmness from below of the function $h$ one obtains
that
\begin{equation} \label{InequalfValFunc}
  f(x) - f^* \ge h(p) - h(p^*) \ge - L \omega(d(p, p^*)) \quad \forall x \in \Omega(p)
  \quad \forall p \in B(p^*, r_1).
\end{equation}
Fix an arbitrary $x \in \Omega_{\delta}$. If $x \in \Omega$, then $F_{\lambda}(x) = f(x) \ge f^*$ for any 
$\lambda \ge 0$ by the definition of $f^*$. Hence one can suppose that $x \notin \Omega$ or, equivalently, 
$\varphi(x) > 0$.

By condition 3, there exists $p \in U(p^*, r_2)$ such that $x \in \Omega(p)$, and 
$\varphi(x) \ge a \omega(d(p^*, \Omega^{-1}(x)))$. Hence $d(p^*, \Omega^{-1}(x)) < r_2$. Due to the continuity from
the right of the function $\omega$, and the fact that $\varphi(x) > 0$ for any $a_0 \in (0, a)$ there
exists $p_0 \in \Omega^{-1}(x)$ such that $d(p_0, p^*) < r_2$, and
$$
  \frac{1}{a_0} \varphi(x) \ge \omega(d(p_0, p^*)).
$$
Note that $x \in \Omega(p_0)$, since $p_0 \in \Omega^{-1}(x)$. Hence with the use of (\ref{InequalfValFunc}), and 
the fact that $r_2 \le r_1$ one gets that
$$
  f(x) - f^* \ge -L \omega(d(p_0, p^*)) \ge -\frac{L}{a_0} \varphi(x),
$$
which implies that $F_{\lambda}(x) \ge f^*$ for all $\lambda \ge L / a_0$. Consequently, taking into account the facts
that $a_0 \in (0, a)$, and $x \in \Omega_{\delta}$ were chosen arbitrary we arrive at the required result.  
\end{proof}

\begin{theorem} \label{ThOmCalmMargFuncNessForExPen}
Let $X$ be a topological space, and $\omega \colon \mathbb{R}_+ \to \mathbb{R}_+$ be a given function. Suppose that
there exist $a > 0$ and $r > 0$ such that
\begin{equation} \label{PhiParamUpperEstim}
  \varphi(x) \le a \omega(d(p, p^*)) \quad \forall x \in \Omega(p) \quad \forall p \in B(p^*, r).
\end{equation}
Then for the penalty function $F_{\lambda}$ to be exact it is necessary that the optimal value function $h$ is
$\omega$-calm from below at $p^*$.
\end{theorem}

\begin{proof}
Let $F_{\lambda}$ be an exact penalty function. Then there exists $\lambda \ge 0$ such that
$$
  F_{\lambda}(x) \ge f^* \quad \forall x \in A.
$$
Applying (\ref{PhiParamUpperEstim}) and the fact that $f^* = h(p^*)$ one gets that for any $p \in B(p^*, r)$
$$
  f(x) - h(p^*) = f(x) - f^* \ge - \lambda \varphi(x) \ge - \lambda a \omega(d(p, p^*)) \ \quad 
  \forall x \in \Omega(p).
$$
Taking the infimum over all $x \in \Omega(p)$ one has
$$
  h(p) - h(p^*) \ge - \lambda a \omega(d(p, p^*)) \quad \forall p \in B(p^*, r),
$$
i.e. $h$ is $\omega$-calm at $p^*$.  
\end{proof}

\begin{remark} \label{Rmrk_LinDependOnParam_Global}
Let, as in Remark~\ref{RmrkLinDependOnParam}, $Y$ be a normed space, $P = Y$, and 
$$
  \Omega = \{ x \in X \mid 0 \in \Phi(x) \}, \quad \Omega(p) = \{ x \in X \mid 0 \in \Phi(x) - p \},
$$
where $\Phi \colon X \rightrightarrows Y$ is a set-valued mapping with closed values. 
Define $p^* = 0$, and $\varphi(x) = d(0, \Phi(x))$. Then $\Omega^{-1}(x) = \Phi(x)$, and, as it is easy to verify, one
has
$$
  \varphi(x) = d(p^*, \Omega^{-1}(x)) \quad \forall x \in X, \quad
  \varphi(x) \le d(p, p^*) \quad \forall x \in \Omega(p) \; \forall p \in P.
$$
Furthermore, for any $\delta > 0$ one has
$$
  \Omega_{\delta} = \bigcup_{p \in U(0, \delta)} \Omega(p).
$$
Therefore by the theorems above the penalty function $F_{\lambda}$ is exact if and only if the optimal value function
is calm from below at the origin, and $F_{\mu}$ is bounded below on $A$ for some $\mu > 0$. Thus, if 
a perturbation of a problem is ``linear'', then the calmness from below of the optimal value function of this problem
is, in essence, equivalent to the exactness of the linear penalty function for this problem.
\end{remark}

Utilizing the same ideas as in Remark~\ref{RmrkParamViaPhi} one can give another characterization of exact penalty
functions.

\begin{theorem}
Let $X$ be a topological space. For any $p \ge 0$ define the function
$$
  h(p) = \inf\{ f(x) \mid x \in A \colon \varphi(x) \le p \}.
$$
Then for the penalty function $F_{\lambda}$ to be exact it is necessary and sufficient that $F_{\mu}$ is bounded from
below on $A$ for some $\mu > 0$, and $h$ is calm from below at the origin, i.e.
$$
  \liminf_{p \to + 0} \frac{h(p) - h(0)}{p} > - \infty.
$$
\end{theorem}

Let us study a connection between the $\omega$-calmness from below of the optimal value function, and the
$\omega$-calmness of the problem $(\mathcal{P}_{p^*})$ at its globally optimal solutions.

It is easy to check that the following result holds true.

\begin{proposition}
Let $X$ be a topological space, $\omega \colon \mathbb{R}_+ \to \mathbb{R}_+$ be a given function, and 
$x^* \in \Omega$ be a globally optimal solution of the problem $(\mathcal{P}_{p^*})$. If the optimal value function $h$
is $\omega$-calm from below at $p^*$, then the problem $(\mathcal{P}_{p^*})$ is $\omega$-calm at $x^*$.
\end{proposition}

In some particular cases, the converse statement is also true. The proposition below is a simple generalization of 
Theorem 2.5 from \cite{LowerOrderCalmness}.

\begin{proposition} \label{Prp_OptimValFunc_OmegaCalmness}
Let $A$ be a compact set, $f$ and $\varphi$ be l.s.c. on $A$, and $\omega \colon \mathbb{R}_+ \to \mathbb{R}_+$ be a
strictly increasing and continuous from the right function such that $\omega(0) = 0$, $\omega(t) > 0$ for any $t > 0$.
Suppose that for any globally optimal solution $x^*$ of the problem $(\mathcal{P}_{p^*})$ there exist $a(x^*) > 0$ and a
neighbourhood $V(x^*)$ of $x^*$ such that
\begin{equation} \label{PenTerm_FirstPertCond}
  \varphi(x) \ge a(x^*) \omega( d(p^*, \Omega^{-1}(x)) ) \quad \forall x \in V(x^*).
\end{equation}
Suppose also that there exist $a > 0$ and $r > 0$ such that
\begin{equation} \label{PenTerm_SecondPertCond}
  \varphi(x) \le a \omega( d(p, p^*) ) \quad \forall x \in \Omega(p) \quad \forall p \in B(p^*, r).
\end{equation}
Then the optimal value function $h$ is $\omega$-calm from below at $p^*$ if and only if the problem 
$(\mathcal{P}_{p^*})$ is $\omega$-calm at each of its globally optimal solutions.
\end{proposition}

\begin{proof}
Let the problem $(\mathcal{P}_{p^*})$ be $\omega$-calm at each of its globally optimal solutions. Then with the use of
the sufficient conditions for the penalty function to be exact at an optimal solution of the problem
$(\mathcal{P}_{p^*})$ in terms of the $\omega$-calmness of this problem (Theorem~\ref{ThCalmLocalPen}) one gets that
$F_{\lambda}$ is exact at every globally optimal solution of the problem $(\mathcal{P}_{p^*})$. Hence by 
Theorem~\ref{ThExPenFuncCompactSet} the penalty function $F_{\lambda}$ is exact. Then applying
Theorem~\ref{ThOmCalmMargFuncNessForExPen} one gets that the optimal value function $h$ is $\omega$-calm from below at
$p^*$. Thus, the proof is complete.  
\end{proof}

\begin{remark}
{(i) Note that for any $p \in P$ and $x \in \Omega(p)$ the inequality $d(p^*, \Omega^{-1}(x)) \le d(p, p^*)$ holds
true. Therefore conditions (\ref{PenTerm_FirstPertCond}) and (\ref{PenTerm_SecondPertCond}) in the proposition above are
consistent. In particular, if the penalty term $\varphi$ has the same form as in Remarks~\ref{RmrkLinDependOnParam} and
\ref{Rmrk_LinDependOnParam_Global}, then conditions (\ref{PenTerm_FirstPertCond}) and (\ref{PenTerm_SecondPertCond})
are satisifed with $a(x^*) = a = 1$.
}

\noindent{(ii) In the above proposition, the assumption that $A$ is compact can be replaced by the assumptions that $X$
is a finite dimensional normed space, and the set $\{ x \in A \mid F_{\mu}(x) < f^* \}$ is bounded for some 
$\mu > 0$. Indeed, one should simply use Theorem~\ref{ThFinDimGenCase} instead of Theorem~\ref{ThExPenFuncCompactSet}.
}

\noindent{(iii) Proposition~\ref{Prp_OptimValFunc_OmegaCalmness} can be proved under different assumptions. Namely, let
$A$ be a compact set, $f$ and $\varphi$ be l.s.c. on $A$, and $\omega \colon \mathbb{R}_+ \to \mathbb{R}_+$ be a given
function. Suppose that for any $x \in A \setminus \Omega$ there exist $L(x) > 0$, $r(x) > 0$ and a neighbourhood
$V(x)$ of $x$ such that
$$
  \varphi(y) \le L(x) \omega(d(p, p^*)) \quad \forall y \in V(x) \cap \Omega(p) 
  \quad \forall p \in B(p^*, r(x)).
$$
Then arguing in a similar way to the proof of Theorem 2.5 from \cite{LowerOrderCalmness} one can verify that the
optimal value function $h$ is $\omega$-calm from below at $p^*$ if and only if the problem $(\mathcal{P}_{p^*})$ is
$\omega$-calm at every of its globally optimal solutions. We leave the proof of this result to the interested reader.
}
\end{remark}

\section{Local Minima of Penalty Functions}
\label{SectLocMinOfPenFunc}

The main goal of the exact penalty method is to solve a constrained optimization problem by constructing an
unconstrained optimization problem that has the same globally optimal solutions as the original problem. However,
usually, optimization algorithms can find only points of local minimum (or even only stationary points) of a nonconvex
function. Therefore, from the practical point of view, it is important to find conditions under which there are no local
minimizers (stationary points) of a penalty function outside the set of feasible points. These conditions ensure that an
optimization algorithm, when applied to a penalty function, indeed, finds a point of local minimum (or at least a
stationary point) of the initial problem. Some of such conditions were studied in 
\cite{Demyanov, DiPilloGrippo, DiPilloGrippo2, ExactBarrierFunc, Ye}. It is worth mentioning that any condition of this
type is usually very restrictive, since it normally consists of the assumption that some sort of a constraint
qualification holds at every infeasible point.

Let $(X, d)$ be a metric space. In this section, we present some simple sufficient conditions ensuring that
there are no stationary points of the penalty function $F_{\lambda}$ outside the set of feasible points $\Omega$. To
this end, recall that if $x \in X$ is a point of local minimum of $F_{\lambda}$, then 
$F_{\lambda}^{\downarrow}(x) \ge 0$. Therefore any point $x \in X$ such that $F_{\lambda}^{\downarrow}(x) \ge 0$ is
called \textit{an inf-stationary point} (or \textit{a lower semistationary point}, see~\cite{Giannessi}) of
$F_{\lambda}$.

\begin{remark}
It is easy to see that if $F_{\lambda}$ is the standard $\ell_1$ penalty function for a nonlinear programming problem,
then the inequality $F_{\lambda}^{\downarrow}(x) \ge 0$ holds true iff the Karush--Kuhn--Tucker conditions hold at the
point $x$. 
\end{remark}

The main result that we present here is rather restrictive, but it is a typical example of the conditions that ensure
that there are no inf-stationary points of the penalty function $F_{\lambda}$ outside $\Omega$.

\begin{proposition} \label{PrpStPointsPenFunc}
Let $f$ be Lipschitz continuous on $A \setminus \Omega$ with Lipschitz constant $L \ge 0$, and $\varphi$ be l.s.c.
on $A$. Suppose also that there exists $a > 0$ such that $\varphi_A^{\downarrow}(x) \le - a$ for any 
$x \in A \setminus \Omega$. If $x \in A$ is an inf-stationary point of $F_{\lambda}$ and $\lambda > L / a$, then 
$x \in \Omega$. In other words, there are no inf-stationary points of $F_{\lambda}$ outside $\Omega$ for any 
$\lambda > L / a$.
\end{proposition}

\begin{proof}
Let $x \in A \setminus \Omega$, and $\overline{a} \in (0, a)$ be arbitrary. Then $\varphi^{\downarrow}_A(x) \le -a$, and
by the definition of rate of steepest descent there exists a sequence $\{ x_n \} \subset A$ converging to $x$ such that
\begin{equation} \label{RofSteepDescEstim}
  \varphi(x_n) - \varphi(x) \le - \overline{a} d(x_n, x) \quad \forall n \in \mathbb{N}.
\end{equation}
Moreover, taking into account the facts that $\varphi(x) > 0$, and $\varphi$ is l.s.c. on $A$ one can suppose
that $\varphi(x_n) > 0$ or, equivalently, $x_n \in A \setminus \Omega$ for all $n \in \mathbb{N}$. Therefore applying
the Lipschitz continuity of $f$ on $A \setminus \Omega$ and (\ref{RofSteepDescEstim}) one gets
$$
  F_{\lambda}(x_n) - F_{\lambda}(x) \le (L - \overline{a} \lambda) d(x_n, x) \quad \forall n \in \mathbb{N},
$$
which yeilds $F_{\lambda}^{\downarrow}(x) \le L - \overline{a} \lambda$. Consequently, 
for any $\lambda > L / \overline{a}$ one has $F_{\lambda}^{\downarrow}(x) < 0$, which implies the desired result due to
the fact that $\overline{a} \in (0, a)$ was chosen arbitrary.  
\end{proof}

In the case when $X$ is a finite dimensional normed space, one can give a little more convenient for applications
formulation of the proposition above, since in this formulation the inequality $\varphi^{\downarrow}(x) \le -a$ is
supposed to hold only locally. 

Denote by $\interior \Omega$ the interior of the set $\Omega$.

\begin{corollary}
Let $X$ be a finite dimensional normed space, $A$ be a closed set, $Q \subset A$ be a given set, $\varphi$ be l.s.c.
on $A$, and $f$ be Lipschitz continuous on any bounded subset of $A \setminus \interior \Omega$. Suppose also that 
for any $x \in A \setminus (\interior \Omega \cup Q)$ there exist $r > 0$ and $a > 0$ such
that $\varphi_A^{\downarrow}(y) \le - a$ for any $y \in B(x, r) \cap (A \setminus \Omega)$. Then for any bounded subset
$C$ of the set $A$ and for any open set $U$ such that $Q \subset U$ there exists $\lambda_0 \ge 0$ such that if 
$x \in C$ is an inf-stationary point of $F_{\lambda}$ and $\lambda \ge \lambda_0$, then $x \in \Omega \cup U$.
\end{corollary}

\begin{proof}
Let $C \subset A$ be a bounded set, and $U$ be an open set such that $Q \subset U$. Then there exists $R > 0$ such that
$C \subset A \cap B(0, R)$. Recall that $X$ is a finite dimensional space. Therefore the set
$$
  K = (A \cap B(0, R)) \setminus (\interior \Omega \cup U)
$$
is compact. Applying the compactness of the set $K$, and the assumptions on the function $\varphi$ one can easily verify
that there exists $a > 0$ such that $\varphi_A^{\downarrow}(x) \le - a$ for any $x \in (A \cap B(0, R)) \setminus
(\Omega \cup U)$. Then a fortiori $\varphi_A^{\downarrow}(x) \le - a$ for any $x \in C \setminus (\Omega \cup U)$. Hence
repeating the proof of Proposition~\ref{PrpStPointsPenFunc} one gets the required result.  
\end{proof}

\begin{remark}
Let us explain the meaning of the corollary above in the case of the classical mathematical programming problem.
Let $F_{\lambda}$ be the $\ell_1$ penalty function for the nonlinear programming problem
$$
  \min f(x) \quad \text{subject to} \quad h_i(x) = 0, \quad g_j(x) \le 0, \quad
  i \in I, \: j \in J,
$$
i.e. $\varphi(x) = \sum_{i = 1}^n |h_i(x)| + \sum_{j = 1}^m \max\{ g_j(x), 0 \}$. Here $I = \{ 1, \ldots, n \}$,
$J = \{ 1, \ldots, m \}$, and all functions $f$, $h_i$ and $g_j$ are continuously differentiable on $\mathbb{R}^d$.
Then one can show that if the Mangasarian--Fromowitz constraint qualification (MFCQ) is satisfied at a point
$x \notin \Omega$, then $\varphi^{\downarrow}(x) < 0$. Moreover, from \cite{HanMangasarian}, Theorem 2.2 (see also
\cite{DiPilloGrippo}, Lemma 3.1) it follows that in this case there exist $a > 0$ and a neighbourhood $V$ of $x$ such
that $\varphi^{\downarrow}(y) \le - a$ for all $y \in V \setminus \Omega$. 

Denote by $Q$ the set of all those $x$ at which MFCQ does not hold true. Then by the corollary above for any bounded
subset $C$, and for any open set $U$ such that $Q \subset U$ there exists $\lambda_0 \ge 0$ such that if $x \in C$ is an
inf-stationary point of $F_{\lambda}$, and $\lambda \ge \lambda_0$, then either $\varphi(x) = 0$ or $x \in U$. In other
words, if $x$ is an inf-stationary point of the penalty function $F_{\lambda}$ for some sufficiently large 
$\lambda \ge 0$, then either $x$ is feasible or $x$ belongs to an arbitrary small neighbourhood of a point at which 
MFCQ is not satisfied.
\end{remark}

\section*{Acknowledgements}

The author wishes to express sincere gratitude and thanks to the late professor V.F.~Demyanov. His advice,
encouragement, as well as, inspiring lectures had a great influence on the author's decision of becoming a professional
mathematician. In addition, the idea of writing this article arose from the author's desire to refine some of the very
interesting results of professor V.F.~Demyanov on exact penalty functions.

\bibliographystyle{abbrv}  
\bibliography{ExactLinearPenFunc_bibl}

\end{document}